\documentclass[11pt]{amsart}
\usepackage[dvips]{graphicx} 
\usepackage{amssymb,amscd,color,latexsym,epsfig,xypic,hyperref,leftindex}
\usepackage[mathscr]{euscript}
\usepackage{tikz-cd}
\usepackage{tikz}
\usetikzlibrary{calc,trees,positioning,arrows,chains,shapes.geometric,%
    decorations.pathreplacing,decorations.pathmorphing,shapes,%
    matrix,shapes.symbols}

\tikzset{
>=stealth',
  punktchain/.style={
    rectangle,
    rounded corners,
    draw=black, thick,
    minimum height=3em,
    text centered,
    on chain},
  line/.style={draw, thick, <-},
  element/.style={
    tape,
    top color=white,
    bottom color=blue!50!black!60!,
    minimum width=8em,
    draw=blue!40!black!90, very thick,
    text width=10em,
    minimum height=3.5em,
    text centered,
    on chain},
  every join/.style={->, thick,shorten >=1pt},
  decoration={brace},
  tuborg/.style={decorate},
  tubnode/.style={midway, right=2pt},
}

\DeclareFontFamily{OT1}{rsfs}{}
\DeclareFontShape{OT1}{rsfs}{n}{it}{<-> rsfs10}{}
\DeclareMathAlphabet{\curly}{OT1}{rsfs}{n}{it}

\DeclareFontFamily{U}{mathb}{\hyphenchar\font45}
\DeclareFontShape{U}{mathb}{m}{n}{
      <5> <6> <7> <8> <9> <10> gen * mathb
      <10.95> mathb10 <12> <14.4> <17.28> <20.74> <24.88> mathb12
      }{}
\DeclareSymbolFont{mathb}{U}{mathb}{m}{n}

\makeatletter
\newcommand{\eqnum}{\refstepcounter{equation}\textup{\tagform@{\theequation}}}
\makeatother

\renewcommand\;{\hspace{.6pt}}
\newcommand\C{\mathbb C}

\newcommand\R{\mathbb R}

\newcommand\Z{\mathbb Z}
\newcommand\ZZ{\mathbb{Z}[\textstyle{\frac 12}]}
\renewcommand\AA{\mathbb A}

\newcommand\LL{\mathbb L}

\renewcommand\P{\mathbf P}

\newcommand\cO{\mathcal O}

\newcommand\cC{\mathcal C}

\newcommand\cG{\mathcal G}

\newcommand\cK{\mathcal K}

\newcommand\cS{\mathcal S}

\newcommand\cY{\mathcal Y}

\newcommand\Fd{F_\bullet}

\newcommand\T{\mathsf{T}}

\renewcommand\({\big(}
\renewcommand\){\big)}
\renewcommand\[{\big[}
\renewcommand\]{\big]}
\newcommand\wt{\widetilde}

\makeatletter
\newcommand{\so}{\ \ext@arrow 0359\Rightarrowfill@{}{\hspace{3mm}}\ }
\makeatother
\newcommand{\rt}[1]{\xrightarrow{\ #1\ }}
\newcommand\To{\longrightarrow}

\newcommand\into{\hookrightarrow}
\newcommand\INTO{\ \ar@{^(->}[r]<-.2ex>}

\newcommand\onto{\to\hspace{-3mm}\to}
\newcommand\Onto{\longrightarrow\hspace{-5.5mm}\longrightarrow}
\newcommand\Mapsto{\ \longmapsto\ }

\newcommand\take{\!\smallsetminus\!}

\newfont{\bigtimesfont}{cmsy10 scaled \magstep5}
\newcommand{\bigtimes}{\mathop{\lower0.9ex\hbox{\bigtimesfont\symbol2}}}
\renewcommand\={\ =\ }

\DeclareMathSymbol{\lefttorightarrow}{3}{mathb}{"FC}
\DeclareMathSymbol{\righttoleftarrow}{3}{mathb}{"FD}

\newcommand\vir{\operatorname{vir}}

\newcommand\loc{\operatorname{loc}}

\newcommand\id{\operatorname{id}}

\newcommand\Hom{\operatorname{Hom}}

\newcommand\Ext{\operatorname{Ext}}

\newcommand\Spec{\operatorname{Spec}}

\newcommand\Cone{\operatorname{Cone}}

\newcommand\tot{\operatorname{tot}}
\newcommand\MF{\mathrm{MF}}
\newcommand\mf{\mathrm{mf}}
\newcommand\Coh{\mathrm{Coh}}

\newcommand\beq[1]{\begin{equation}\label{#1}}
\newcommand\eeq{\end{equation}}
\newcommand\beqa{\begin{eqnarray*}}
\newcommand\eeqa{\end{eqnarray*}}

\newcommand\arXiv[1]{\href{http://arxiv.org/abs/#1}{arXiv:#1.}}
\newcommand\mathAG[1]{\href{http://arxiv.org/abs/math/#1}{math.AG/#1.}}

\makeatletter \@addtoreset{equation}{section} \makeatother
\renewcommand{\theequation}{\thesection.\arabic{equation}}
\newtheorem{defn}[equation]{Definition}
\newtheorem{thm}[equation]{Theorem}
\newtheorem{ex}[equation]{Example}
\newtheorem*{thm*}{Theorem}
\newtheorem{lem}[equation]{Lemma}
\newtheorem{cor}[equation]{Corollary}
\newtheorem{prop}[equation]{Proposition}

\theoremstyle{definition}
\newtheorem{rmk}[equation]{Remark}

\title{Lagrangian classes in $K$-theory}
\author{Dongwook Choa and Jeongseok Oh} 

\begin{document}

\maketitle

\begin{abstract} 
For a $(-1)$-shifted Lagrangian in a critical locus, we construct a homomorphism from the $K$-group of matrix factorisations of the critical locus to the $K$-group of the Lagrangian, partially answering the Joyce-Safronov conjecture \cite{JS}. The key step is the construction of a specialisation functor for categories of matrix factorisations along the deformation to the normal cone.

Any $(-2)$-shifted symplectic space is a $(-1)$-shifted Lagrangian of a point, whose $K$-group is $\Z$. The image of $1\in \Z$ under the above homomorphism is the virtual structure sheaf \cite[Definition 5.9]{OT1}.

We prove that two equivalent critical models of a given critical locus induce homomorphisms that commute via Kn\"orrer periodicity. When a torus acts on the Lagrangian, we further prove a localisation formula, namely the commutativity of the homomorphisms associated with the Lagrangian and its fixed locus.
\end{abstract}

\setcounter{tocdepth}{1}
\tableofcontents

\section{Introduction}
\subsection*{Joyce-Safronov conjecture}
Kontsevich-Soibelman \cite{KS} conjectured that for a $(-1)$-symplectic space ($(-1)$-symplectic for simplicity) $M$, there is an associated `category of matrix factorisations' $\MF(M)$. For instance, we associate $\MF(M):=\MF(U,f)$ if $M$ is given as the critical locus of a function $f$ on a smooth space $U$. We call the pair $(U,f)$ a {\em critical model} of $M$ in this case to emphasise that $M$ comes equipped with $(-1)$-symplectic structure given by $(U,f)$.

Heuristically, Joyce-Safronov conjecture then asserts that for a $(-1)$-Lagrangian $L\to M$, there exits a functor 
$$
\MF(M)\To D(L):=D(\mathrm{Coh}L)
$$ 
to the derived category of coherent sheaves on $L$. For instance, if $L$ is $(-2)$-symplectic, then it is $(-1)$-Lagrangian of a point which has a natural $(-1)$-symplectic structure given by the critical model $(\Spec\C,0)$. Hence we expect to have $\MF(\Spec\C,0)\to D(L)$. Assuming a suitable boundedness, the conjecture guarantees an existence of a homomorphism 
$$
\Z\cong K\(\MF(\Spec\C,0)\)\To K_0(L)
$$ 
between $K$-groups. A minimal requirement of the conjecture is that this morphism sends $1$ to the virtual structure sheaf $\widehat{\cO}^{\vir}_L\in K_0\(L,\ZZ\)$ \cite[Definition 5.9]{OT1}. 

\smallskip
Our main theorem regards this $K$-theoretic Joyce-Safronov conjecture when $M$ is given as a critical locus. We say a (-1) Lagrangian $L\to M$ {\em oriented} if the relative tangent complex $T_{L/U}$ is oriented in the sense of \cite[Equation (59)]{OT1} -- we see that $T_{L/U}\cong T_{L/U}^\vee[-2]$ in Section \ref{sect:normalform} so that we can say its orientability. 

\begin{thm*}[Theorem \ref{main2}] Let $L$ be an oriented $(-1)$-Lagrangian of $(-1)$-symplectic $M$ which is the critical locus of a critical model $(U,f)$. Suppose that 
\begin{itemize}
\item $L$ is a quasi-compact Artin stack with affine diagonal, 
\item $L\to M$ is a DM morphism,
\item $L$ has the resolution property,
\item the relative virtual dimension of $L\to U$ is even. 
\end{itemize}
Then we construct a homomorphism, called the {\em $K$-theoretic Lagrangian class}, 
\beq{Kmain}
K\(\MF(U,f)\)\To K_0\(L,\ZZ\). 
\eeq
If $(U,f)=(\Spec\C,0)$, then it takes $\cO_{\Spec\C}$ to $\widehat{\cO}^{\vir}_L$\footnote{In \cite{Ku}, Kuhn constructed a lift of $\widehat{\cO}^{\vir}_L$ in the $\Z/2$-graded derived category $D^{\Z/2}(L)$ under a choice of spin structure on $T_L$. It provides a functor at the categorical level for the Joyce-Safronov conjecture in the case $(U,f)=(\Spec\C,0)$.}.
\end{thm*}

To see where the difficulty of the Theorem is hidden, we would like to explain what is going on with a local cutout model of $L$.

\subsection*{Cutout model}
Following the cutout model of Joyce-Safronov \cite[Example 3.6]{JS}, we assume that $L$ is isomorphic to the zero locus $Z(s)$, where
\begin{itemize}
\item $s:\cO_V\to E$ is a section of
\item $E$ an orthogonal bundle with the quadratic form $q:E\cong E^*$ on
\item $V$ a smooth stack with smooth map $p:V\to U$,
\item satisfying $q(s,s)=p^*f$,
\end{itemize}
\beq{-1local}
\xymatrix@R=1mm@C=2mm{
&&&& (E,q) \ar[lld] \\
L\ \ 
=\ \; Z(s)\ \ \ \ddto \stackrel{\mathrm{closed}}{\subset}\hspace{-1mm} && V \ar[dd]^-{p}\ar@/_1pc/[urr]_-{s} && \\
&&&&& q(s,s)=p^*f \\
M\ \ \cong\ \ Z(df)\ \ \ \stackrel{\mathrm{closed}}{\subset}\hspace{-1mm}  && U \ar[rrd]^-f \\
&&&&\C .\\
}
\eeq
Over the total space $\mathrm{tot}E$, the quadratic form $q$ can be considered as a quadratic function, so that we can consider the category $\MF(E,q)$ of matrix factorisations. We have a canonical module $\cS_E$, called the Clifford module, of the Clifford algebra of $(E,q)$\footnote{By a {\em canonical module}, we mean a module whose $K$-theory class equals the {\em prescribed class}; see Corollary \ref{corsE}. To define it, we need a spin structure on $E$. Thus, it may not exist in the absence of a spin structure, and even when one exists, it need not be unique, as it depends on the choice of spin structure. On the other hand, it always exists locally. However the $K$-theory class is well-defined without a spin structure once we introduce the coefficients $\ZZ$.}. Together with the tautological section $\tau$ of $E$, the pair $(\cS_E,\tau)$ defines a tautological object in $\MF (E,-q)_{V}$, which gives a functor
$$
\MF (E,q)\rt{(\cS_E,\tau)\otimes }\MF (E,0)_{V}.
$$
Using the section $s:\cO_V\to E$, we have the following diagram,
\begin{align}\label{p1}
\xymatrix@C=13mm{
&\MF (E,q) \ar[r]^-{(\cS_E,\tau)} \ar[d]_-{s^*}&\MF (E,0)_{V} \ar[d]^-{s^*}& \\
\MF (U,f)\ar[r]^-{\widehat{p}^*}&\MF(V,p^*f) \ar[r]^-{(\cS_E,s)}&\MF (V,0)_{Z(s)}\ar[r]^-{h^+\oplus h^-[1]} & D^{\Z/2}_{Z(s)}(V).
}
\end{align}
In the above,
\begin{itemize}
\item $\MF(E,0)_V$ denotes the category of factorisations supported on the zero section $V\subset E$,
\item $\widehat{p}^*$ denotes the twisted pullback $\surd{\det T^*_{V/U}}\otimes p^*(-)$,
\item $(\cS_E,s)$, which we will call the Clifford factorisation, is a factorisation given by the section $s$,
\item $h^{\pm}$ takes cohomology.
\end{itemize}
Explicitly, it maps $\Fd\in \MF(U,f)$\vspace{-0.5mm} to the cohomology of $(\cS_E,s)\otimes \surd{\det T^*_{V/U}}\otimes p^*\Fd$. When $(U,f)=(\Spec \C,0)$, it takes $\cO_{\Spec\C}$\vspace{-0.2mm} to the cohomology of $(\cS_E,s)\otimes \surd{\det T^*_{V}}$. We would like to see that its $K$-theory class is $\widehat{\cO}^{\vir}_{Z(s)}$ after assuming that the orthogonal bundle $E$ splits into $E=\Lambda\oplus \Lambda^*$ with $q$ the pairing of $\Lambda$. It defines the Koszul $2$-periodic complex
$$
\(\Lambda^{\bullet}\Lambda^*,s\)=\{\cdots \rt{s} \Lambda^{\mathrm{even}}\Lambda^*\rt{s} \Lambda^{\mathrm{odd}}\Lambda^* \rt{s} \Lambda^{\mathrm{even}}\Lambda^*\rt{s} \cdots\}
$$
in $\MF(V,0)_{Z(s)}$. By its definition which we give later, $(\cS_E,s)$ differs from it by a twist by $\surd{\det \Lambda^*} \cong \surd{\det T^*_V\otimes \surd{\det T_{L}}}$. As studied in \cite{PV:A, KO, OS, OT1}, the $K$-theory class of the cohomology of $(\Lambda^{\bullet}\Lambda^*,s)\otimes \surd{\det T_L^*}$\vspace{-0.5mm}, which is $(\cS_E,s)\otimes \surd{\det T^*_{V}}$, equals $\widehat{\cO}^{\vir}_{Z(s)}$.

\subsection*{Globalisation} A typical idea of gluing $Z(s)$ is to replace the ambient space $V$ by the cone $i:C:=C_{Z(s)/V}\into E$. Then the diagram \eqref{p1} can be replaced by another one using $C$ instead of $V$,
\begin{align}\label{p2}
\xymatrix@C=13mm{
&\MF (E,q) \ar[r]^-{(\cS_E,\tau)} \ar[d]_-{i^*}&\MF (E,0)_{V} \ar[d]^-{i^*}& \\
\MF (U,f)\ar@{-->}[r]^-{\exists\ ?}&\MF(C,q|_C) \ar[r]^-{(\cS_E,\tau)}&\MF (C,0)_{Z(s)}\ar[r]^-{h^+\oplus h^-[1]} & D^{\Z/2}_{Z(s)}(V),
}
\end{align}
except for the first functor $\MF(U,f)\to \MF(C,q|_C)$ in the bottom row. Constructing this first functor for globalisation was our main difficulty to have the theorem. Together with the twisted pullback $\widehat{p}^*:\MF(U,f)\to \MF(V,p^*f)$, the specialisation
$$
\mathrm{sp}:\MF(V,p^*f)\To \MF(C,q|_C)
$$ 
which we will construct gives a solution of this problem. 

\subsection*{Specialisation}
To state our solution in a general setup, we consider the following situation. Let $Z\to Y$ be a morphism between Artin stacks and let $f$ be a function on $Y$. We say that the function $f$ lies in $I_{Z/Y}^2$\vspace{-0.2mm} if for any\footnote{\label{footnote3}This condition is guaranteed around $x\in Z$ if $f\in I^2_{Z/A}$\vspace{-0.5mm} for {\em some} $A\ni x$. Suppose there are two such local embeddings $Z\subset A_1, A_2$ and $f\in I^2_{Z/A_1}$\vspace{-0.5mm}. We would like to prove $f\in I^2_{Z/A_2}$\vspace{-0.3mm}. By replacing $A_1$ by $A_1\times A_2$ if necessary, we may assume that $A_1$ smoothly surjects to $A_2$ by $\pi:A_1\to A_2$ and $\pi$ takes $Z$ to $Z$ identically. Then we can check that $\pi^{-1}I^2_{Z/A_2}=I^2_{Z/A_1}\cap \pi^{-1}\cO_{A_2}$, implying $f\in I^2_{Z/A_2}$.} local embedding $Z\subset A$ into $A$ smooth over $Y$, the pullback of $f$ to $A$ lies in $I_{Z/A}^2$, where $I_{Z/A}\subset \cO_A$\vspace{-0.5mm} is the ideal shaef of $Z\subset A$. The condition $f\in I_{Z/Y}^2$ guarantees that the function $\lambda^{-2}f$ is defined on the deformation $M^o_{Z/Y}$\vspace{-0.5mm} from $Y$ to the intrinsic normal cone $C:=C_{Z/Y}$.\vspace{-0.5mm} Here $\lambda$ is the deformation parameter, $M^o_{Z/Y}\to \Spec \C[\lambda]$\vspace{-0.5mm}. By abusing notations, we denote by $\lambda^{-2}f$ the pullback function on $C$. 

\begin{thm*}[Theorem \ref{thm: spfunctor}] Let $Z\to Y$ be a morphism from a quasi-compact stack $Z$ with affine diagonal to a smooth stack $Y$; and let $f$ be a function on $Y$ lying in $I_{Z/Y}^2$. Then we can construct a triangulated functor
$$
\mathrm{sp}: \MF (Y,f)\To \Coh (C,\lambda^{-2}f).
$$
Here $\MF$ denotes the category of {\em locally free} factorisations whereas $\Coh$ is the {\em coherent} one. {\em (}\!We explain the precise statement after setting concrete notations in Section \ref{sect:sp}.{\em )}
\end{thm*}

Using the Theorem, the diagram \eqref{p2} can be modified to
\begin{align}\label{p3}
\xymatrix@C=13mm{
&\MF (E,q) \ar[r]^-{(\cS_E,\tau)} \ar[d]_-{i^*}&\MF (E,0)_{V} \ar[d]^-{i^*}& \\
\MF (U,f)\ar[r]^-{\mathrm{sp}\;\circ\; \widehat{p}^*}&\Coh(C,q|_C) \ar[r]^-{(\cS_E,\tau)}&\Coh (C,0)_{Z(s)}\ar[r]^-{h^+\oplus h^-[1]} & D^{\Z/2}_{Z(s)}(V).
}
\end{align}
Then taking $K$-groups of the bottom glues to give the homomorphism \eqref{Kmain}.

\subsection*{Comparison}
Let $(Q,q_Q)$ be a $SO(2m,\C)$ bundle on $U$. The quadratic form $q_Q$ defines a quadratic function on the total space $\mathrm{tot}\;Q$, and the pullback function $f$ is also defined on it. Then the two critical models $(U,f)$ and $\(\mathrm{tot}\;Q,f+q_Q\)$ give the same $(-1)$-symplectic space $M$. Hence we have two Lagrangian classes when a $(-1)$-Lagrangian $L\to M$ is given. The equivalence $\MF(U,f)\to \MF (\mathrm{tot}\;Q,f+q_Q)$, called the Kn\"orrer periodicity, allows us to compare the two Lagrangian classes.

\begin{thm*}[Theorem \ref{thm:critind}] For an even rank quadratic bundle $(Q,q_Q)$ on $U$, the two Lagrangian classes given by different critical models $(U,f)$ and $(\mathrm{tot}\;Q,f+q_Q)$ commute with the Kn\"orrer periodicity,
$$
\xymatrix{
K\(\MF(U, f)\) \ar[r]^{\text{\em Lag}} \ar[d]_{\text{\em Kn\" orrer}}& K_0\(L,\ZZ\)\\
K\(\MF(\mathrm{tot}\;Q, f+q_Q)\). \ar[ur]_{\text{\em Lag}}&
}
$$
In particular, for a $(-2)$-symplectic $L$, the Lagrangian class
$$
K\(\MF(\AA^{2m}, x_1y_1+\cdots +x_my_m)\) \To K_0\(L\)
$$
takes the Clifford factorisation to the virtual structure sheaf $\widehat{\cO}_L\in K_0\(L,\ZZ\)$.
\end{thm*}

\subsection*{Torus localisation}
Now suppose that a torus $\mathsf{T}:=\C^*$ acts on both $L$ and $U$ making $L\to U$ to be equivariant and $f$ to be invariant. We further assume that the $\T$-action preserves the symplectic structure on $T_{L/U}$. In this case, we can construct the $\mathsf{T}$-equivariant Lagrangian class,
$$
\mathrm{Lag}: K\(\MF^{\mathsf{T}}(U, f)\)\To K_0^{\mathsf{T}}(L)_\mathrm{loc}
$$
with the localised $K$-group $K_0^{\mathsf{T}}(L)_\mathrm{loc}$ defined in \cite[Page 54]{OT1}. Also we have a new critical model $(U^{\mathsf{T}},f^{\mathsf{T}})$, $f^{\mathsf{T}}:=f|_{U^{\mathsf{T}}}$, on the fixed loci. We will check that $f^{\mathsf{T}}$ lies in $I_{L^{\mathsf{T}}/U^{\mathsf{T}}}^2$, so that we can define
$$
\mathrm{Lag}^{\mathsf{T}}: K\(\MF^{\mathsf{T}}(U^{\mathsf{T}}, f^{\mathsf{T}})\)\To K_0^{\mathsf{T}}\(L^{\mathsf{T}}\)_\mathrm{loc}.
$$
Let $\mathfrak{e}_{\mathsf{T}}$ and $\surd\mathfrak{e}_{\mathsf{T}}$ denote the equivariant $K$-theoretic Euler class and square root Euler class defined in \cite[Page 56]{OT1}, respectively; and let $T_{L/U}^{\mathrm{mov}}$ be the moving part of the tangent complex $T_{L/U}$. Then we prove the following localisation theorem under the assumption that $L$ is a {\em scheme}.

\begin{thm*}[Theorem \ref{main4}]
Let $i_U:U^{\T}\into U$ and $i_L:L^{\T}\into L$ denote the embeddings. Then we have the following commutative diagram 
\begin{equation*}
\xymatrixcolsep{11pc}
\xymatrix{
K\(\MF^{\mathsf{T}}(U, f)\) \ar[r]^-{\mathrm{Lag}} \ar[d]_-{i_U^*} & K_0^{\mathsf{T}}(L)_\mathrm{loc} \\
K\(\MF^{\mathsf{T}}(U^{\mathsf{T}}, f^{\mathsf{T}})) \ar[r]^-{\frac{\mathrm{Lag}^{\mathsf{T}}}{\mathfrak{e}_{\T}(N_{U^{\T}/U})\cdot \sqrt{\mathfrak {e}_{\mathsf{T}}}(T_{L/U}^\mathrm{mov})}}  & K_0^{\mathsf{T}}\(L^{\mathsf{T}}\)_\mathrm{loc} . \ar[u]_-{i_{L*}}
}
\end{equation*}
In particular, when $(U,f)=(\Spec\C,0)$, this reproves the ordinary localisation theorem \cite[Theorem 7.3]{OT1}.
\end{thm*}

\subsection*{Other works} 
Cao-Toda-Zhao \cite{CTZ} constructed a homomorphism from $K(\MF(U,f))$ to $K_0(L,\ZZ)$ using a similar strategy, with one notable difference in the specialisation procedure. Rather than specialising factorisations directly, they first pass to coherent sheaves on the zero locus $Z(f)$ via Orlov's functor \cite{EP, Or}, and then apply the usual specialisation to $Z(q)\subset C$ at the level of $K$-groups. After pushing forward to $E$, they showed that perfect complexes vanish so that the construction becomes well-defined and yields elements in $K(\MF(E,q))$. The authors then discussed applications to GSLM and $K$-theoretic DT4 theory. We expect that our specialisation functor in Theorem \ref{thm: spfunctor} agrees with, at the level of $K$-groups, the morphism constructed in \cite{CTZ} after pushforward along $C\into E$. Establishing this agreement would be interesting, as it amounts to a nontrivial compatibility between specialisation and Orlov’s functor.

Khan-Kinjo-Park-Safronov \cite{KKPS} have announced the constructible version of Joyce-Safronov conjecture, namely, there is a morphism of sheaves 
$$
\varphi_M|_L\To \omega_L
$$
from the sheaf of vanishing cycles on $M$ to the dualising sheaf of $L$. Taking cohomology gives a morphism from DT cohomology to the Borel-Moore homology
$$
H_{\mathrm{DT}}^*(M):= \mathbb{H}(M,\varphi_M)\To H^{BM}_*(L)\cong \mathbb{H}(L,\omega_L).
$$
Note that they did not assume $M$ to be critical locus, so their construction is fully global (for both $L$ and $M$).

\subsection*{Acknowledgements}
The authors are deeply indebted to Richard Thomas for indispensable contributions and extensive discussions that were crucial to the development of the ideas and results of this paper. 

Also we would like to thank Hyeonjun Park for several helpful suggestions that improved the formulation of the paper. We are indebted to Yalong Cao, Yukinobu Toda, and Gufang Zhao for generously sharing their earlier draft, and are particularly grateful to Yalong Cao for his kind and insightful correspondence. J. O. wishes to thank Mark Shoemaker for his valuable ~advice.

Both authors were supported by the National Research Foundation of Korea (NRF) grant funded by the Korean government (MSIT)(RS-2024-00339364). D.C was supported by the National Research Foundation of Korea(NRF) grant funded by the Korea government(MSIT)(RS-2026-25478462). J. O was supported by the New Faculty Startup Fund from Seoul National University and the POSCO science fellowship.

\section{Specialisation of matrix factorisations}\label{sect:sp}

\subsection{Matrix factorisations}\label{sect:MF}
Given a space with function, factorisations of the function form a DG category. Once we choose a suitable notion of acyclicity, we obtain its derived category as the Verdier quotient by acyclic objects. We will work with the DG category of factorisations we started with, rather than a DG enhancement of the derived category. We introduce the precise notions below. 

For a function $f: Y \to \mathbb A^1$ on an Artin stack $Y$, its \emph{factorisation} consists of a pair of $\cO_Y$-modules $(F_+, F_-)$ with morphisms $d_\pm: F_\pm \to F_\mp$, called the {\em structure maps}, such that $d_+\circ d_- = f\!\cdot\!\id_{F_-}$ and $d_-\circ d_+ = f\! \cdot \!\id_{F_+}$. We often denote it by 
$$
\Fd := 
\xymatrix{
F_+\ar@/_1pc/[r]_-{d_+} & F_- \ar@/_1pc/[l]_-{d_-}
}.
$$
Morphisms between factorisations form a $(\Z/2)$-graded complex
$$
\Hom^\bullet (\Fd^1 , \Fd^2) := \left(\Hom_{Y}^\bullet (F_\pm^1, F_\pm^2), D = [d_\pm, -] \right) \in \mathrm{Ch^{\Z/2}}.
$$
This yields a ($\Z/2$)-graded DG category $\mathrm{Fact}(Y, f)$. It is pretriangulated with the shift:
$$
\Fd := 
\xymatrix{
F_+\ar@/_1pc/[r]_-{d_+} & F_- \ar@/_1pc/[l]_-{d_-}
}
\rightsquigarrow\ \Fd[1] := 
\xymatrix{
F_-\ar@/_1pc/[r]_-{-d_-} & F_+ \ar@/_1pc/[l]_-{-d_+}
} 
$$
and the mapping cone of a closed even morphism: 
$$
\left \{
\begin{array}{c}
a: \Fd^1 \to \Fd^2\\
D(a) = 0
\end{array} 
\right \}
\ \rightsquigarrow\ \Cone(a) := 
\xymatrix@C=5mm{
F_+^2 \oplus F_-^1 \ar@/_1.2pc/[rr]_-{\begin{bsmallmatrix} d^2_+ & a_- \\ 0 &-d^1_-\end{bsmallmatrix}
} && F_-^2 \oplus F_+^1  \ar@/_1.2pc/[ll]_-{\begin{bsmallmatrix} d^2_- & a_+ \\ 0 &-d^1_+\end{bsmallmatrix}}
} 
$$
Its homotopy category, denoted by $[\mathrm{Fact}(Y,f)]$, is canonically triangulated. 
\begin{defn}
\label{defn: MFandmf}
A factorisation $\Fd$ is called a matrix factorisation if $F_\pm$ are locally free sheaves of finite rank. It forms a full subcategory
$$
\MF(Y, f)\ \subset\ \mathrm{Fact}(Y, f),
$$
called the category of matrix factorisations. Similarly, a factorisation $\Fd$ is called a coherent (resp. quasi-coherent) factorisation if $F_\pm$ are coherent (resp. quasi-coherent) sheaves. The category of coherent (resp. quasi-coherent) factorisations
$$
\Coh(Y, f)\ \subset\ \mathrm{Fact}(Y, f), \quad (\text{resp.} \;\ \mathrm{Qcoh}(Y, f)\ \subset\ \mathrm{Fact}(Y, f))
$$
is a full DG subcategory generated by coherent (resp. quasi-coherent) factorisations. We denote their homotopy categories by $[\MF(Y, f)]$,  $[\Coh(Y, f)]$ and $[\mathrm{Qcoh}(Y, f)]$ respectively. 
\end{defn}

We move on to the notion of \emph{derived categories}. Before going into the precise definitions, we explain the ideas. Since a factorisation $\Fd$ is not a complex, the usual notion of acyclicity (via its cohomology) does not make sense. Instead, we consider a \emph{short exact sequence} of factorisations 
$$
0 \To \Fd^1 \rt{a} \Fd^2 \rt{b} \Fd^3 \To 0.
$$
By its definition, both $a$ and $b$ are closed of even degree, and they induce short exact sequences of sheaves on each parity. Then its \emph {totalisation} which is defined to be
$$
\mathrm{Tot}=\mathrm{Tot} \left( \Fd^1 \xrightarrow{a} \Fd^2 \xrightarrow{b} \Fd^3 \right) \cong \Cone((b, 0): \Cone(a) \to \Fd^3)
$$ 
has the following property: $\Hom(G_\bullet, \mathrm{Tot})$ is acyclic for any $G_\bullet$. In other words, the Yoneda module of the totalisation is quasi-isomorphic to zero.  This lead to the following notion of acyclic objects. 

\begin{defn}[\cite{CFGKS, EP, Po}]
\label{defn: absacyc}
We say a matrix (resp. coherent, quasi-coherent) factorisation $\Fd$ is \emph{absolutely acyclic} if it is contained in the smallest thick pretriangulated subcategory generated by the totalisations of short exact sequences of matrix (resp. coherent, quasi-coherent) factorisations. We denote such subcategories by 
$$
\mathrm{abs}(\MF(Y, f)) \subset \MF(Y, f) 
$$ 
(resp. $\mathrm{abs}(\Coh(Y, f)) \subset \Coh(Y, f)$, $\mathrm{abs}(\mathrm{Qcoh}(Y, f)) \subset \mathrm{Qcoh}(Y, f)$).
\end{defn}

\begin{defn}
The \emph{absolute derived category of matrix factorisations} is defined as a Verdier quotient
$$
D^{\mathrm{abs}}\MF(Y, f) := [\MF(Y, f)] / [\mathrm{abs}(\MF(Y, f))]
$$
Similarly, the \emph{absolute derived category of coherent (resp. quasi-coherent) factorisations} is defined as a Verdier quotient
$$
D^{\mathrm{abs}}\Coh(Y, f) := [\Coh(Y, f)] / [\mathrm{abs}(\Coh(Y, f))]
$$
(resp. $D^{\mathrm{abs}}\mathrm{Qcoh}(Y, f) := [\mathrm{Qcoh}(Y, f)] / [\mathrm{abs}(\mathrm{Qcoh}(Y, f))]$).
\end{defn}
The natural inclusions $\MF(Y, f) \subset \Coh(Y, f) \subset \mathrm{Qcoh}(Y, f)$ induce fully faithful embeddings between the derived categories \cite[Corollary 2.3 (i),(k)]{EP},
\beq{Eqeq}
D^{\mathrm{abs}}\MF(Y, f) \hookrightarrow D^{\mathrm{abs}}\Coh(Y, f) \hookrightarrow D^{\mathrm{abs}}\mathrm{Qcoh}(Y, f).
\eeq
In the following, we define a flexible enlargement of $\Coh(Y,f)$ which we will use throughout the paper.
\begin{defn}
We denote by
$$
\mf(Y,f) \subset \mathrm{Qcoh}(Y, f)
$$ 
the full DG subcategories generated by objects whose images in $D^{\mathrm{abs}}\mathrm{Qcoh}(Y, f)$ are contained in the essential image of $D^{\mathrm{abs}}\Coh(Y, f)$. We define a subcategory 
$$
\mathrm{abs}(\mf(Y,f)):=\mf (Y,f)\cap \mathrm{abs}(\mathrm{Qcoh}(Y,f))
$$ 
of absolutely acylic objects and its derived category $D^{\mathrm{abs}}\mf(Y,f)$ as a Verdier quotient $[\mf (Y,f)]/\mathrm{abs}(\mf(Y,f))$. 
\end{defn}
The enlargement $\Coh(Y, f) \subset \mf(Y,f)$ leaves its derived category unchanged, as it induces a quasi-equivalence 
$$
D^{\mathrm{abs}}\Coh(Y, f) \cong D^{\mathrm{abs}}\mf(Y, f).
$$
Indeed, \eqref{Eqeq} tells us that $D^{\mathrm{abs}}\Coh(Y, f)$ is a full subcategory of $D^{\mathrm{abs}}\mf(Y, f)$, and the inclusion is essentially surjective. On the other hand, $\mf(Y, f)$ allows a quasi-coherent representation of its object which makes it more convenient to work with.

There is yet another notion of acyclicity that we will use throughout the paper. We say a factorisation $\Fd$ is {\em locally contractible} if there is an open cover $\{U_i\}$ of $Y$ such that $\Fd\vert_{U_i}$ is null-homotopic. The following lemma shows that the two notions are equivalent for matrix factorisations, allowing us to freely switch between them in the rest of the paper.
\begin{lem}
\label{lem: coeqcon}
The followings are equivalent for a matrix factorisation $\Fd$:
\begin{enumerate}
\item $\Fd$ is absolutely acyclic.
\item $\Fd$ is locally contractible.
\end{enumerate}
\end{lem}
\begin{proof} 
($1\Rightarrow 2$)
Suppose $\Fd = \mathrm{tot}(\Fd^1 \to \Fd^2 \to \Fd^3)$ is the totalisation of a short exact sequence. After passing to a sufficiently small open cover, we can assume 
$$
F^2_\pm = F^1_\pm \oplus F^3_\pm, \quad d^2_\pm = 
\begin{pmatrix} d^1_{\pm} & 0 \\ \phi_\pm & d^3_{\pm} \end{pmatrix}.
$$ 
where $\phi_\bullet: \Fd^3 \to \Fd^1[1]$ is the morphism defining the extension. Then the totalisation $\Fd=\mathrm{tot}(\Fd^1 \to \Fd^2 \to \Fd^3)$ becomes  
\begin{equation*}
\Fd=\Fd^1[1] \oplus \Fd^1 \oplus \Fd^3 \oplus \Fd^3[1] ,\quad  d=\begin{pmatrix}
d^1[1] & 0 & 0& 0\\ 
1& d^1 & 0 & 0\\
0 &\phi_\bullet & d^3 & 0\\
0 & 0 & 1 & d^3[1]
\end{pmatrix}.
\end{equation*}
It is null-homotopic with explicit contracting homotopy 
$$h=
\begin{pmatrix}
0 & 1 & 0& 0\\ 
0& 0 & 0 & 0\\
0 &0 & 0 & 1\\
0 & 0 & 0 & 0
\end{pmatrix}:\Fd\to \Fd[1].
$$
Local contractibility is stable under taking shifts, mapping cones, and direct sums. It also descends to its summand. Indeed, if $\Fd = \Fd^1 \oplus \Fd ^2$ with a contracting homotopy $h=h_{ij}, (i, j = 1, 2)$, then it is easy to check that $h_{ii}$ is a contracting homotopy of $\Fd^i$ for $i=1,2$.

($2\Rightarrow 1$)
The notion of coacyclicity, which is equivalent to the absolute acyclicity for matrix factorisations by \cite[Section 2.4]{EP}, is local in a smooth topology \cite[Proposition 2.2.6]{CFGKS}. Therefore we may assume that $\Fd$ is globally contractible by the homotopy $h$. Then the morphism
$$
\Fd \xrightarrow{h\oplus 1} \mathrm{Cone}(\Fd \rt{1} \Fd)\cong \mathrm{Tot}(0 \rt{0} \Fd \rt{1} \Fd )
$$
splits. Therefore $\Fd$ is absolutely acyclic.
\end{proof}

\subsection*{Support} Let $Z \subset Y$ be a closed substack. We say a matrix factorisation $\Fd$ is {\em supported on}  $Z$ if it is absolutely acyclic on $Y \take\, Z$. Matrix factorisations supported on $Z$ form a pretriangulated subcategory of $\MF(Y, f)$, which we denote by $ \MF(Y, f)_Z$. Similarly, we can define a pretriangulated subcategory $\mf(Y,f)_Z\subset \mf(Y,f)$.

\subsection*{Pullback} Let $\pi: X \to Y$ be a morphism of stacks. We can define the (derived) pullback $\mf(Y,f)\to \mf(X,\pi^*f)$ when $Y$ has enough $\pi^*$-adapted objects. We don't have to derive it when we work with matrix factorisations. In this case, we always have the pullback
\begin{equation*}
\pi^* : \MF(Y, f) \To \MF(X, \pi^*f),  \quad \Fd\! \Mapsto\! \pi^*\Fd,
\end{equation*}
which preserves the pretriangulated structure. It takes the subcategory $\MF(Y,f)_Z$ to $\MF(X,\pi^*f)_{Z\times_Y X}$. In particular, when $Z=\emptyset$, $\pi^*$ takes absolute acyclic objects to absolute acyclic objects.

When $\pi$ is flat, we can also define the pullback
\begin{equation*}
\pi^* : \mf(Y, f)_Z \To \mf(X, \pi^*f)_{Z\times_Y X},  \quad \Fd\! \Mapsto\! \pi^*\Fd,
\end{equation*}
without deriving it since any object in $\mf (Y,f)$ is $\pi^*$-adapted. It is also pretriangulated, and preserves acyclicity.

\subsection*{Tensor product}
Given a matrix factorisation $G_\bullet\in \MF(Y,g)_Z$ supported on $Z$, the tensor product gives a pretriangulated functor
\begin{equation*}
-\otimes G_\bullet: \mf(Y, f)_W \To \mf(Y, f+g)_{Z\cap W}, \quad \Fd \Mapsto \Fd \otimes G_\bullet. 
\end{equation*}
In particular, when $W=\emptyset$, it preserves absolute acyclicity.

\subsection*{Pushforward}
Let $X$ be a quasi-compact Artin stack with affine diagonal. Then a proper morphism $\pi:X\to Y$ to an Artin stack $Y$ defines the (derived) pushforward on quasi-coherent factorisations using \v Cech resolutions. Following the idea of \cite[Proposition 10.6.2]{DG}, this pushforward restricts to the subcategories $\mf\subset \mathrm{Qcoh}$,
$$
R\pi_*: \mf (X,\pi^*f)\To \mf (Y,f),
$$
and preserves the absolute acylicity \cite[Section 3.5]{EP}. Throughout the paper, we omit the symbol $R$ and write $\pi_*$ in place of $R\pi_*$.

\subsection*{$K$-groups} Let $K(\MF(Y, f))$ denote the Grothendieck group of its derived category $D^{\mathrm{abs}}\MF(Y, f)$. A pretriangulated DG functor preserving the absolute acyclicity induces a homomorphism between $K$-groups,
$$
\MF(Y, f) \To \MF(Y', f')\ \rightsquigarrow\ K(\MF(Y, f)) \To K(\MF(Y', f')).
$$ 
Similar story holds for $K(\mf(Y, f))$. In particular, the pullback, tensor product and pushforward functors introduced above induce the morphisms of $K$-groups.

\subsection{Specialisation}\label{Sect:sp}
Now let $Z\to Y$ be a morphism from a quasi-compact Artin stack $Z$ with affine diagonal to a smooth Artin stack $Y$. Let $f$ be a function on $Y$ lying in $I_{Z/Y}^2$. Recall from the Introduction that $f\in I_{Z/Y}^2$\vspace{-0.3mm} means that for any local closed embedding $Z\subset A$ with $A$ smooth over $Y$, the pullback of $f$ lies in $I_{Z/A}^2$. 

We consider the deformation $M^o\to \AA^1=\Spec \C[\lambda]$ from $Y$ to the normal cone $C$. The condition $f\in I_{Z/Y}^2$\vspace{-0.3mm} guarantees that $q:=f/\lambda^2$ becomes a function on $M^o$. By abusing notation, we also denote by $q$ its restriction to the central fibre $C$. Then we would like to construct a functor
\beq{YMo}
\MF(Y,f)\To \mf(C,q),
\eeq
using the space $M^o$.

The example below illustrates what we are going to do na\"ively.

\subsection*{Example} Suppose that $Z=\{0\}\subset Y=\C^2_{x,y}$ and $f=xy$. Then one can associate to a factorisation
$$
\xymatrix{
\cO_{\C^2}\ar@/_1pc/[r]_-{d_+=x} & \cO_{\C^2} \ar@/_1pc/[l]_-{d_-=y} \hspace{-5mm} & \in \quad \MF(Y,f),
}
$$
another one 
$$
\xymatrix{
\cO_{M^o}\ar@/_1pc/[r]_-{d_+=x/\lambda} & \cO_{M^o} \ar@/_1pc/[l]_-{d_-=y/\lambda} \hspace{-5mm} & \in \quad \MF(M^o,q),
}
$$
as $x/\lambda$, $y/\lambda$ are functions on $M^o$. We obtain a factorisation in $\MF (C,q)$ by restricting it to $C$. This looks to suggest \eqref{YMo} can be taken by replacing $d_\pm$ by $d_\pm/\lambda$. But if we consider a different example $(d_+,d_-)=(xy,1)$, then the suggestion no longer works since $1/\lambda$ is not a function on $M^o$. So we need a remedy of this suggestion. However, we emphasise that the precise construction is based on this simple idea.

\subsection*{Object to object via graph construction}  
A key idea of the remedy (to the suggestion in the above Example) is to take the blowup of $M^o$ along the locus where $d_\pm/\lambda$ are not defined. Although this depends on the object $\Fd$, the blowup locus always lies entirely in the central fibre $C$. Recall that a blowup may be described as the closure of a rational section of a projective bundle. Following this idea, we describe our blowup as the closure of a rational section of a Grassmannian bundle over $M^o$. Then the advantage is that we can use the tautological bundles to associate to $\Fd\in \MF(Y,f)$ the specialisation $\mathrm{sp}(F_\bullet)\in \mf (C,q)$. Now we provide a precise construction.

Set $r_+$ and $r_-$ be the ranks of $F_+$ and $F_-$ respectively. Considering them as bundles on $M^o$, we let $G$ be the product of two Grassmannians $Gr(r_+,F_+\oplus F_-)\times_{M^o} Gr(r_-,F_+\oplus F_-)$ over $M^o$. We have a rational section of $G$ defined on $Y\times (\AA^1\take\{0\})\subset M^o$,
\beq{lce}
Y\times (\AA^1\take\{0\})\ \into\ G,\ \ (x,\lambda)\Mapsto \left(\Gamma_{\frac{d_+(x)}{\lambda}}, \Gamma_{\frac{d_-(x)}{\lambda}}, \lambda\right) 
\eeq
given by the graphs of $d_+/\lambda$ and $d_-/\lambda$. It is the dotted arrow in the below diagram,
$$
\xymatrix@R=12mm{
& G|_{\lambda\neq 0} \ar[d] \ar@{^(->}[r] & G \ar[d]  \\
Y\times (\AA^1\take\{0\}) \ar@{=}[r] \ar@{^(->}[ru]^-{\Gamma_{d/\lambda}} \ar@{-->}[rru]_-{\eqref{lce}}& Y\times (\AA^1\take\{0\}) \ar@{^(->}[r]& M^o .
}
$$ 
We let 
\beq{MFd}
M^o_{\Fd}\subset G\quad\text{and}\quad C_{\Fd}\subset G|_{\lambda=0}
\eeq
be the closure (or equivalently, the scheme theoretic image) of \eqref{lce} and its restriction to $\lambda=0$, respectively. Hence $M^o_{\Fd}$ is the closure of the image inside $G$. Since $Y$ is smooth, the closure $M^o_{\Fd}$ is integral. We have the projection morphisms $M^o_{\Fd}\to M^o$ and $C_{\Fd}\to C$, which are proper. 

\begin{ex} \label{ex1}Suppose $f=0$. Consider the factorisation $\Fd:=\cO_Y$ of $f=0$ with $F_+:=\cO_Y$, $F_-:=0$. By convention, we set $Gr(0,F_+\oplus F_-):=Y$. Then we have $G\cong M^o\cong M^o_{\Fd}$ and $C\cong C_{\Fd}$. 
\end{ex}

Over $M^o_{\Fd}$, we have the pullbacks of tautological bundles $\xi_+$ and $\xi_-$ from $G$. We would like to make it to a factorisation $\xi_\bullet \in \MF(M^o_{\Fd},q)$. To do so, we consider two composition maps for each parity,
$$
\alpha_{\pm}:\xi_\pm\subset F_+\oplus F_-\xrightarrow[1\oplus \;q]{q\;\oplus 1} F_+\oplus F_-\onto (F_+\oplus F_-)/\xi_\mp.
$$
This means that $\alpha_+$ is defined using $q\oplus 1$ in the middle map, while $\alpha_-$ uses $1\oplus q$ according to the parity. On the common zero $Z(\alpha):=Z(\alpha_+)\cap Z(\alpha_-)$, the maps $\xi_\pm\to F_+\oplus F_-$ factor through $\xi_\pm\to \xi_{\mp}$. Hence these define a subfactorisation $\xi_\bullet$ of
\beq{framefact}
\xymatrix{
F_+\oplus F_-\ar@/_1pc/[r]_-{q\;\oplus\; 1} & F_+\oplus F_- \ar@/_1pc/[l]_-{1\;\oplus\; q} &\hspace{-8mm} \text{in }\ \MF(Z(\alpha),q),
}
\eeq
cf. \cite[Section 2]{Oh}. So the following Lemma tells us that we obtain a factorisation $\xi_\bullet \in \MF(M^o_{\Fd},q)$.
\begin{lem}
\label{lem: vanish}
The morphisms $\alpha_\pm$ vanish on $M^o_{\Fd}$.
\end{lem}
\begin{proof}
We compute $\alpha_+$ first at $\lambda\neq 0$. The first embeddings $\xi_+\subset F_+\oplus F_-$ defining $\alpha_+$ is
$$
\xi_+\cong F_+\ \xrightarrow{1\oplus\; \frac{d_+}{\lambda}}\ F_+\oplus F_-.
$$
Composing it with $F_+\oplus F_-\rt{q\;\oplus 1} F_+\oplus F_-$ becomes
$$
\xi_+\cong F_+\ \xrightarrow{q\;\oplus\; \frac{d_+}{\lambda}}\ F_+\oplus F_-.
$$
The image lies in the graph of $d_-/\lambda$ because $q=f/\lambda^2=d_-/\lambda\!\circ\!\; d_+\lambda$, which proves $\alpha_+\!=0$ on $M^o_{\Fd}|_{\lambda\neq 0}$. Since $M^o_{\Fd}$ is the closure, $\alpha_+=0$ on here. Similarly we can prove $\alpha_-\!=0$.
\end{proof}

Pulling $\xi_\bullet\in \MF(M^o_{\Fd},q)$ back to $\MF (C_{\Fd},q)$ via central fibre $C_{\Fd}\into M^o_{\Fd}$ and pushing it down to $\mf (C,q)$ via the proper morphism $\rho: C_{\Fd}\to C$, we obtain a specialisation 
\beq{defsp}
\mathrm{sp}(F_\bullet):=\rho_*\xi_\bullet \in \mf (C,q).
\eeq
Note that the pushforward $\mf (C_{\Fd},q)\to \mf (C,q)$ is defined once we specify the \v Cech resolutions of objects in $\mf (C_{\Fd},q)$. Note that this is the only place where we use the condition on $Z$; quasi-compact with affine diagonal. Hence, for each $\Fd$, we have an a priori chosen resolution of $\xi_\bullet \in \MF (C_{\Fd},q)$ to define $\mathrm{sp}(F_\bullet) \in \mf (C,q)$. One requirement for these a priori chosen \v Cech resolutions is that they are shifts of each other for $\Fd$ and $\Fd[1]$. We achieve this requirement using the fact that $M_{\Fd}$ is isomorphic to $M_{\Fd[1]}$.

\begin{ex} \label{ex2} The factorisation $\cO_Y$ in Example \ref{ex1} has the specialisation $\mathrm{sp}(\cO_Y)=\cO_{C}$. Here, since $C_{\Fd}=C$, we don't need to resolve $\cO_{M^o}$ to get $\mathrm{sp}(\cO_Y)=\cO_{C}$.
\end{ex}

\subsection*{Morphism to morphism}
Now we associate to a morphism $a: \Fd^1\to \Fd^2$ in $\MF (Y,f)$ its specialisation $\mathrm{sp}(a): \mathrm{sp}(F^1_\bullet)\to \mathrm{sp}(F^2_\bullet)$ in $\mf (C,q)$. 

We first construct it for an even morphism 
$$
a=(a_{++}, a_{--}) \in \Hom(F_+^1, F_+^2) \oplus \Hom(F_-^1, F_-^2).
$$
For each parity, consider morphisms of tautological bundles on the fibre product $M^o_{\Fd^1}\times_{M^o} M^o_{\Fd^2}$, 
{\setlength{\jot}{-1pt}
\begin{align} \label{defofA}
& a'_{++}:\; \xi^1_+\ \into F^1_+\oplus F^1_-\rt{A_{++}} F^2_+\oplus F^2_-\Onto (F^2_+\oplus F^2_-)/\xi^2_+,\\ 
& a'_{--}:\; \xi^1_-\ \into F^1_-\oplus F^1_+\rt{A_{--}} F^2_-\oplus F^2_+ \Onto (F^2_-\oplus F^2_+)/\xi^2_-, \nonumber
\end{align}
}%
where $A_{\pm\pm}$ are defined using $a$,
\beq{matrixAA}
A_{++} = 
\begin{pmatrix}a_{++} & 0 \\ 
\frac{1}{\lambda}D(a)_{+-} & a_{--} \end{pmatrix}, \quad 
A_{--} = \begin{pmatrix}  a_{--} & 0 \\ 
\frac{1}{\lambda}D(a)_{-+}  &  a_{++} \end{pmatrix}.
\eeq
Here, $D = [d, -]$ is the differential of morphism of complexes of factorisations, that is, $D(a)_{+-}=d^2_+ \circ a_{++} -a_{--}\circ d^1_+$ and $D(a)_{-+}=d^2_- \circ a_{--} -a_{++}\circ d^1_-$. The computation $A_{++}\(v, d^1_+(v)/ \lambda\) = \( a_{++}(v), d^2_+\circ a_{++}(v)/\lambda\)$ checks that the map $a'_{++}$  is zero on the closure $M^o_{\Fd^1,\Fd^2}$ of the graphs $\Gamma_{d^1_{\pm}/\lambda}$, $\Gamma_{d^2_{\pm}/\lambda}$ inside the fibre product 
$$
M^o_{\Fd^1,\Fd^2}\subset M^o_{\Fd^1}\times_{M^o} M^o_{\Fd^2}.
$$
Similar computation shows $a'_{--}$ also vanishes on $M^o_{\Fd^1,\Fd^2}$\vspace{-0.5mm}. Therefore, $A_{\pm\pm}$ define a morphism of factorisations $A_{\pm\pm}:\xi^1_\bullet\to \xi^2_\bullet$ on $\MF (M^o_{\Fd^1,\Fd^2}, q)\vspace{-0.5mm}$. Its pullback to $\lambda=0$ and pushforward to $C$ gives a morphism in $\mf (C,q)$.  

To regard the resulting one as a morphism between $\mathrm{sp}(\Fd^1)$ and $\mathrm{sp}(\Fd^2)$, we must carefully define the pushforward $\mf(C_{\Fd^1,\Fd^2},q)\to \mf(C,q)$, where $C_{\Fd^1,\Fd^2}$ denotes the central fibre of $M^o_{\Fd^1,\Fd^2}$\vspace{-0.7mm}. The key point is that it suffices to specify \v Cech resolutions of the pullbacks $\xi^1_\bullet$, $\xi^2_\bullet\in \mf (C_{\Fd^1,\Fd^2},q)$\vspace{-0.5mm}. These are obtained by pulling back the chosen \v Cech resolutions of $\xi^1_\bullet\in \mf (C_{\Fd^1},q)$ and $\xi^2_\bullet\in \mf (C_{\Fd^2},q)$\vspace{-0.5mm}, respectively. Under the pushforward to $\mf (C,q)$, these resolutions map precisely to $\mathrm{sp}(\Fd^1)$ and $\mathrm{sp}(\Fd^2)$. Moreover, the morphism between $\xi^1_\bullet$ and $\xi^2_\bullet$ constructed from $a$ pulls back to a morphism between the corresponding resolutions. This allows us to define $\mathrm{sp}(a):\mathrm{sp}(\Fd^1)\to \mathrm{sp}(\Fd^2)$.

For an odd morphism $b = (b_{+-}, b_{-+}) \in \Hom(F_+^1, F_-^2) \oplus \Hom(F_-^1, F_+^2),$  we similarly define morphisms of tautological bundles
{\setlength{\jot}{-1pt}
\begin{align} \label{defofB}
& b'_{+-}:\; \xi^1_+\ \into F^1_+\oplus F^1_- \rt{B_{+-}} F^2_-\oplus F^2_+\Onto (F^2_-\oplus F^2_+)/\xi^2_-,\vspace{-1.5mm}\\
& b'_{-+}:\; \xi^1_-\ \into F^1_-\oplus F^1_+ \rt{B_{-+}} F^2_+\oplus F^2_-\Onto (F^2_+\oplus F^2_-)/\xi^2_+, \nonumber
\end{align}
}%
for each parity, where $B_{\pm\mp}$ are defined using $b$,
\beq{matrixBB}
B_{+-} = \begin{pmatrix}  b_{+-} & 0 \\ 
\frac{1}{\lambda}D(b)_{++} & -b_{-+} \end{pmatrix}, \quad 
B_{-+} = \begin{pmatrix}  b_{-+} & 0 \\ 
\frac{1}{\lambda}D(b)_{--}  & -b_{+-} \end{pmatrix}.
\eeq
The rest of the construction goes exactly the same. 

This defines $\mathrm{sp}(a)$,
$$
\mathrm{sp}(a) : \mathrm{sp}(F^1_\bullet) \To \mathrm{sp}(F^2_\bullet)\quad\text{in}\quad \mf (C,q).
$$

\subsection*{Differential to differential}
For an even morphism $a$, we would like to prove that
\beq{prdiff}
\mathrm{sp}(D(a))=D(\mathrm{sp}(a)).
\eeq
Since $a$ is of even degree, $D(a)$ should be of odd degree. Let $DA_{\pm\mp}$ denote the matrix $B_{\pm\mp}$ in \eqref{defofB} for $D(a)$ rather than $b$, 
$$
\xi^1_\pm\ \into F^1_\pm\oplus F^1_{\mp}\rt{DA_{\pm\mp}} F^2_\mp\oplus F^2_{\pm}\Onto (F^2_\mp\oplus F^2_\pm)/\xi^2_\mp,
$$
defining $\mathrm{sp}(D(a))$. Using \eqref{matrixBB}, we can check that
\begin{align*}
DA_{\pm\mp} &= 
\begin{pmatrix}
\frac{1}{\lambda }D(a)_{\pm\mp} & 0 \\
\frac{1}{\lambda^2}D^2(a)_{\pm\pm} & -\frac{1}{\lambda}D(a)_{\mp\pm}
\end{pmatrix} \\
&=
\begin{pmatrix}
\frac{1}{\lambda}D(a)_{\pm\mp} & 0 \\
0& -\frac{1}{\lambda}D(a)_{\mp\pm}
\end{pmatrix}
=\begin{pmatrix}
0 & 1 \\
q & 0
\end{pmatrix}A_{\pm\pm}-
A_{\mp\mp}\begin{pmatrix}
0 & 1 \\
q & 0
\end{pmatrix}.
\end{align*}
The resulting morphism defines $D(\mathrm{sp}(a))$ since the \v Cech resolution of $\Fd[1]$ was chosen to be the shift of that of $\Fd$. This checks \eqref{prdiff} when $\deg a=$even. We can check it for $\deg a$=odd similarly.

\subsection*{Identity to identity} For $\id:\Fd\to \Fd$, the space $M^o_{\Fd,\Fd}\subset M^o_{\Fd}\times_{M^o}M^o_{\Fd}$ is the diagonal, and $\id:\Fd\to \Fd$ maps to $\id=\mathrm{sp}(\id):\mathrm{sp}(\Fd)\to\mathrm{sp}(\Fd)$.

\subsection*{Composition to composition}
Given even morphisms $a_1: \Fd^1\to \Fd^2$, $a_2: \Fd^2\to \Fd^3$ in $\MF (Y,f)$, we would like to prove that 
\beq{prcomm}
\mathrm{sp}(a_2\circ a_1)=\mathrm{sp}(a_2)\circ \mathrm{sp}\mathrm(a_1).
\eeq 
For this we consider the fibre produce $M^o_{\Fd^1,\Fd^2}\times_{M^o}M^o_{\Fd^2,\Fd^3}\times_{M^o}M^o_{\Fd^1,\Fd^3}$\vspace{-0.3mm} on which all $a_1$, $a_2$ and $a_2\circ a_1$ defined. Also let $A_1, A_2$ and $A_{12}$ denote the matrices in \eqref{defofA} for $a_1$, $a_2$ and $a_2\circ a_1$, respectively,
$$
F^1_\pm\oplus F^1_\mp\rt{A_{1,\pm\pm}} F^2_\pm\oplus F^2_{\mp} \rt {A_{2, \pm\pm}}  F^3_\pm\oplus F^3_\mp.
$$
Using \eqref{matrixAA}, one can directly check that $A_1 \cdot A_2 = A_{12}$; for example,
\begin{align*}
\(A_{1} \cdot A_{2}\)_{++} &=
\begin{pmatrix}
a_{1++} & 0 \\
\frac{1}{\lambda}D(a_1)_{+-}& a_{1--}
\end{pmatrix}
\begin{pmatrix}
a_{2++} & 0 \\
\frac{1}D{\lambda}(a_2)_{+-}& a_{2--}
\end{pmatrix} \\
&=
\begin{pmatrix}
(a_1\circ a_2)_{++} & 0 \\
\frac{1}{\lambda}\(a_1 \circ D(a_2) + D(a_1)\circ a_{2}\)_{+-}& (a_1\circ a_2)_{--}
\end{pmatrix} 
=(A_{12})_{++}
\end{align*}
Here, we used the graded Leibniz rule for the last equality. To conclude that this proves \eqref{prcomm} for $\deg a_1=\deg a_2$ even, we must specify the \v Cech resolutions used to define the pushforward from the fibre product. As expected, we use the pullback resolutions of $\xi_\bullet^i$ that were already chosen in $\mf (C_{\Fd^i},q)$ when defining $\mathrm{sp}(\Fd^i)$, for $i=1,2,3$. Then the morphisms pull back to define morphisms between the resolutions. The identity $A_1\cdot A_2=A_{12}$ proves \eqref{prcomm} together with this pushforward. We can check \eqref{prcomm} for morphisms of any degree in a similar way.

\subsection*{Absolutely acyclic to absolutely acyclic}
Suppose that $\Fd$ is absolutely acyclic. We would like to prove that its specialisation $\mathrm{sp}(\Fd)$ is also absolutely acyclic. By Lemma \ref{lem: coeqcon}, $\Fd$ is locally contractible. Since the deformation to normal cone $M^o$ and its blowup $M^o_{\Fd}$ are stable under flat pullbacks, especially open immersions, we may assume that $\Fd$ is contractible with a contracting homotopy $h_\bullet = h_\pm : F_\pm \to F_\mp$. Using this, we construct an explicit contracting homotopy of $\xi_\bullet$. We first extend $h_\bullet$ to a contracting homotopy $\wt h_\bullet$ of the factorisation \eqref{framefact}. In the diagram below, the first and third modules are the odd factor of \eqref{framefact} while the second is the even factor. The lower morphisms are the structure maps; the upper ones, the homotopy $\wt h_\bullet$,
$$
\xymatrix{
F_+\oplus F_-\ar@/_1pc/[rrr]_-{1\;\oplus\; q} &&& F_+\oplus F_- \ar@/_1pc/[lll]_-{\tilde h_+= \begin{bsmallmatrix} 1 & -\lambda h_- \\ \lambda h_+ & 0 \end{bsmallmatrix} } \ar@/_1pc/[rrr]_-{q\;\oplus\; 1} &&& F_+\oplus F_- \ar@/_1pc/[lll]_-{\tilde h_-=\begin{bsmallmatrix} 0 & -\lambda h_+ \\ \lambda h_- & 1 \end{bsmallmatrix}.}
}. 
$$ 
To see $\wt h_\bullet$ descends to a contracting homotopy of $\xi_\bullet$, as in Lemma \ref{lem: vanish}, we consider the composition maps
$$
\beta_{\pm}:\xi_\pm \hookrightarrow F_+\oplus F_-\xrightarrow{\tilde h_\pm} F_+\oplus F_-\onto \Fd/\xi_\mp.
$$
We claim that these $\beta_\pm$ vanish on $M^o_{\Fd} \vert_{\lambda \neq 0}$. Writing an element of $\xi_\bullet$ as $(v, d_+(v)/\lambda)$, the homotopy relation $[d_\bullet, h_\bullet] = \id_{\Fd}$ checks 
$$
\beta_+\(v, \frac{d_+(v)}{\lambda}\) = \( v- h_-d_+(v), \lambda h_+(v)\) = \(d_-h_+(v), \lambda h_+(v)\) 
$$
which is zero in $F_+\oplus F_-/\xi_-$. This shows that $\beta_+ =0$ on $M^o_{\Fd}\vert_{\lambda\neq 0}$\vspace{-0.5mm} and hence on its closure $M^o_{\Fd}$. Similarly, $\beta_-$ is also zero on $M^o_{\Fd}$. As a result, $\tilde h_\bullet$ restricts to $\xi_\bullet$ to be a contracting homotopy, making $\xi_\bullet$ contractible. Then again by Lemma \ref{lem: coeqcon}, $\xi_\bullet$ is absolutely acyclic. Recall that $\mathrm{sp}(\Fd)$ is the pullback of $\xi_\bullet$ to $C_{\Fd}$ followed by the pushforward to $C$. Since both functors preserve absolute acyclicity, We conclude $\mathrm{sp}(\Fd)$ is absolutely acyclic. 

\begin{thm}
\label{thm: spfunctor}
The specialisation 
$$
\mathrm{sp}: \MF(Y, f) \To \mf(C, q)
$$ 
is a DG functor preserving absolute acyclicity. In particular, it induces a triangulated functor 
$$
\mathrm{sp}: D^{\mathrm{abs}}\MF(Y, f)\To D^{\mathrm{abs}}\mf(Y, f)\cong D^{\mathrm{abs}}\Coh(Y, f),
$$ 
and a homomorphism between their $K$-groups
$$
\mathrm{sp}: K\(\MF(Y, f)\) \To K\(\mf(C, q)\).
$$
\end{thm}

\subsection{Smooth cover}
We introduce a commutativity of the specialisation and a smooth cover. Let $Z\to Y$ and $f\in I^2_{Z/Y}$\vspace{-0.5mm} be as in the previous section. Suppose, furthermore, that we have an embedding $Z\into A$ with smooth surjective morphism $\pi:A\to Y$ commuting $Z$,
\beq{Scover}
\xymatrix@R=5mm{
Z\ar@{^(->}[r]\ar[rd] & A\ar[d]^-{\pi}\\
& Y.
}
\eeq
Then one can check $\pi^*f\in I^2_{Z/A}$\vspace{-0.5mm}. In fact, we have a morphism of deformations $\pi: M^o_{Z/A}\to M^o_{Z/Y}$ with $\pi^*(f/\lambda^2)=(\pi^*f)/\lambda^2$ by abusing notation for the projection morphism $\pi$. The morphism $\pi$ is smooth (although it is not a base change) so that we can define the pullback $\pi^*$ without resolutions.

We may expect a commutativity in the level of DG category but it requires a further discussion. It is rather simpler to discuss with derived categories, but we state it in the level of $K$-theory which will be enough for our purpose.

\begin{prop}\label{Kcomm} For the smooth cover \eqref{Scover}, we have the following commutative diagram of $K$-groups
\beq{MKMKMK}
\xymatrix{
K\(\MF (Y,f)\)\ar[r]^-{\mathrm{sp}}\ar[d]_-{\pi^*} & K\(\mf (C_{Z/Y},q)\) \ar[d]^-{\pi^*}\\
K\(\MF (A,\pi^*f)\)\ar[r]^-{\mathrm{sp}} & K\(\mf(C_{Z/A},\pi^*q)\)
}
\eeq
\end{prop}
\begin{proof}
For a matrix factorisation $\Fd\in \MF(Y,f)$, we have a fibre diagram
$$
\xymatrix@R=7mm{
M^o_{\pi^*\Fd}\ar[r]^-{\pi}\ar[d]_-r & M^o_{\Fd}\ar[d]^-{r}\\
M^o_{Z/A} \ar[r]^-{\pi} & M^o_{Z/Y}
}
$$
because the blowup loci in $M^o_{Z/A}$ and $M^o_{Z/Y}$ to define $M^o_{\pi^*\Fd}$ and $M^o_{\Fd}$, respectively, is the base change from one to the other; and $\pi$ is flat. It also tells us that the tautological factorisation on $M^o_{\pi^*\Fd}$ is the pullback of the one $\xi_\bullet$ on $M^o_{\Fd}$. We denote it by $\pi^*\xi_\bullet\in \MF (M^o_{\pi^*\Fd},\pi^*q)$. Hence $\pi^*\xi_\bullet \in \MF (C_{\pi^*\Fd},\pi^*q)$ at the central fibre is also the pullback of $\xi_\bullet \in \MF (C_{\Fd},q)$.

In the level of derived category, we have the base-change theorem \cite[Proposition 5.9]{BDFIK},
$$
\xymatrix{
D^{\mathrm{abs}}\MF (C_{\Fd},q) \ar[r]^-{\pi^*} \ar[d]_-{\rho_*} & D^{\mathrm{abs}}\MF (C_{\pi^*\Fd},\pi^*q) \ar[d]^-{\rho_*}  \\
D^{\mathrm{abs}}\mf (C_{Z/Y},q) \ar[r]^-{\pi^*} & D^{\mathrm{abs}}\mf (C_{Z/A}, \pi^*q) 
}
$$
basically because we can use any resolutions for the pushforwards $\rho_*$. Hence we have $\pi^*\rho_*\xi_\bullet\cong \rho_*\pi^*\xi_\bullet\in D^{\mathrm{abs}}\mf (C_{Z/A},\pi^*q)$, and
$$
\pi^*\mathrm{sp}(\Fd)=\pi^*\rho_*\xi_\bullet=\rho_*\pi^*\xi_\bullet=\mathrm{sp}(\pi^*\Fd)\in K\(\mf(C_{Z/A},\pi^*q)\),
$$
which proves the commutativity \eqref{MKMKMK}.
\end{proof}

\section{Clifford factorisation}\label{sect:CF}
In this section, we review a classical theory of Clifford algebras and modules. We refer readers to \cite{Ch} for details. 

\subsection{Clifford algebra/module}\label{Sect:MF}
Let $(E,q)$ be an $SO(2n,\C)$ bundle over an Artin stack $Y$\footnote{\label{f4}In some places, we use arguments from \cite{OT1,OS}. In those instances, we either regard $Y$ as a DM stack or interpret the arguments as a generalisation to the case where $Y$ is an Artin stack.}. Its \emph{Clifford algebra} is defined as
$$
Cl(E,q):= T(E)/(v\otimes v +q(v,v)=0).
$$
It is a bundle on $Y$ of rank $2^{2n}$. It is canonically $(\Z/2)$-graded by the parity of a word-length of its element.

A \emph{Clifford module} of $Cl(E, q)$ is a sheaf $F$ of $\cO_Y$-modules with Clifford multiplication which is an $\cO_Y$-algebra homomorphism $\phi_F: Cl(E, q) \to \mathcal{E}nd_Y(F)$. We require them to be also $(\Z/2)$-graded $F = F_+\oplus F_-$. We call it \emph{irreducible} if each fibre $F\vert_x$ is an irreducible representation of $Cl(E, q)\vert_x$. The Clifford algebra $Cl(\C^{2n}, q)$ at a fibre is isomorphic to the matrix algebra of rank $2^n$. Therefore, an irreducible $F$, if exists, must be locally free of rank $2^n$ with $\phi_F$ isomorphism.  

We associate to a maximal isotropic subbundle $\Lambda\subset E$\footnote{There is always a preferred one on a suitable cover as explained in \cite{EG}. We use such a cover whenever the existence of a maximal isotropic subbundle is needed.} a canonical irreducible Clifford module $\cS_\Lambda$ defined to be the quotient
$$
\cS_\Lambda:=\frac{Cl(E,q)}{\langle \Lambda\rangle},
$$
where $\langle \Lambda\rangle\subset Cl(E,q)$ is the left ideal generated by $\Lambda$. A local decomposition $E\cong \Lambda \oplus \Lambda^*$ gives a local isomorphism $\cS_\Lambda\cong \Lambda^\bullet \Lambda^*$ of $\cO_Y$-modules. In particular, $\cS_\Lambda$ is irreducible and $(\Z/2)$-graded.

\begin{prop}
\label{lem: IandS}
We have an equality {\em (}\!\;as $(\Z/2)$-graded elements{\em )}
$$
[\cS_\Lambda] = [\Lambda^\bullet \Lambda^*] \in K^0(Y). 
$$
In particular, $[\cS_\Lambda]$ is $K$-theoretic Euler class $\mathfrak{e}(\Lambda)$.
\end{prop}
\begin{proof}
We use one-parameter family of $SO(2n,\C)$ bundles,
$$
E_\lambda : = \frac{\ker \(E \oplus (\Lambda\oplus \Lambda^*) \xrightarrow{\pi -\lambda\;\cdot\;\mathrm{pr}_2}  \Lambda^*\)}{\Lambda},\ \ \lambda\in \AA^1,
$$
where $\pi:E\to \Lambda^*$ denotes the cokernel of $\Lambda\subset E$. Its restriction is isomorphic to $E$ at $\lambda\neq 0$ and $E_0\cong \Lambda\oplus \Lambda^*$ at $\lambda=0$. The bundle $\Lambda\subset E_
\lambda$ remains as a maximal isotropic subbundle as
$$
\Lambda\cong \frac{\Lambda\oplus (\Lambda\oplus 0)}{\Lambda}\subset E_\lambda
$$
so that we obtain the family of $S_\Lambda$. This is a degeneration from $S_\Lambda$ in a generic fibre to $\Lambda^\bullet \Lambda^*$ in the central fibre.
\end{proof}

\begin{rmk}
The bundle $\cS_\Lambda$ has the following differential geometric interpretation. Consider the corresponding real, positive definite quadratic bundle $(E_\R, q_\R)$. The maximal isotropic bundle $\Lambda \subset E$ determines a $ q_\R$-compatible complex structure $J$ and hence, a canonical $Spin^c$-structure on $E_\R$. Its associated complex spinor bundle is precisely $\cS_\Lambda$.
\end{rmk}

\subsection*{Some properties}
For the proposition below we assume that $E$ comes equipped with the spin structure which we briefly review in Appendix \ref{AppA}. 
\begin{prop}
\label{prop: sqcorrection}
When $E$ is spin, we can assign the square root determinant line bundle $\sqrt{\det \Lambda}$ to each maximal isotropic subbundle $\Lambda\subset E$ so that the $(\Z/2)$-graded Clifford module
$$
\cS_\Lambda  \otimes \sqrt{\det \Lambda}\[\vert \Lambda \vert\]
$$
is independent of $\Lambda$ up to a canonical isomorphism after shifting by $\vert \Lambda \vert$ the sign of $\Lambda$ determined by the orientation of $E$ \cite[Definition 2.2]{OT1}. In particular, $\vert \Lambda \vert\cdot[\cS_\Lambda \otimes \sqrt{\det \Lambda}]$ represents the $K$-theoretic square root Euler class $\sqrt{\mathfrak{e}}(E)$ in \cite[Definition 5.3]{OT1} by Proposition \ref{lem: IandS}.
\end{prop}

We will prove Proposition \ref{prop: sqcorrection} in Appendix \ref{AppA}. 

When $E$ is spin, we set $\cS_E:=\cS_\Lambda \otimes \sqrt{\det \Lambda}$ for positive $\Lambda$. In fact, $\surd{\det \Lambda}$ exists in $K$-group $K^0(Y,\ZZ)$ with $\ZZ$-coefficients even without spin structure by \cite[Lemma 5.1]{OT1}. Hence the $K$-theoretic class $[S_E]$ is always defined to be $[\cS_\Lambda]\cdot [\surd{\det \Lambda}]$. 
\begin{cor} \label{corsE} Even if $E$ is not spin, we have the equality
$$
[\cS_E] = \sqrt{\mathfrak{e}}(E)\ \in\ K^0\(Y,\ZZ\).
$$
\end{cor}

We sometimes use $\cS_E$ if it were a module even when $E$ does not admit a spin structure. In this case, we work with an actual module $\cS_\Lambda$ and twist it by $\surd{\det \Lambda}$ in $K$-theory.

\subsection*{Isotropic reduction}
Let $K\subset E$ be an isotropic subbundle. Then we have a new $SO$-bundle $(K^\perp/K, \overline q)$, called the reduction by $K$, which defines a new Clifford algebra $Cl(K^\perp/K, \overline q)$. By Corollary \ref{corsE} and \cite[Equation (106)]{OT1}, we have a relationship between the irreducible Clifford modules $\cS_E$ and $\cS_{K^\perp/K}$,
$$
[\cS_E] = [\cS_{K^\perp/K}] \cdot [\mathfrak{e}(K)\otimes \sqrt{\det K}]\ \in\ K^0\(Y,\ZZ\).
$$

\subsection{Clifford factorisation}
Let $s\in H^0(E)$ be a section, and $F$ be a locally free of finite rank Clifford module. Then the Clifford multiplication by $s$ defines a matrix factorisation
\beq{eq: fact}
(F,s):=(\cdots F_+ \xrightarrow{s\cdot} F_- \xrightarrow{s\cdot} F_+ \to \cdots) \in \MF(Y, q(s,s))_{Z(s)}
\eeq
of $q(s,s)\in \mathcal O_Y$. 

\begin{lem}
\label{lem: support}
The matrix factorisation $(F, s)$ \eqref{eq: fact} is supported on $Z(s)$. In particular, when $s$ is isotropic, $(F, s)$ is a periodic complex exact off $Z(s)$.
\end{lem}
\begin{proof}
Passing to an \'{e}tale open subset, we can find a section $\lambda$ of $E$ outside $Z(s)$ such that $q(s, \lambda) =1$. The Clifford relation says $(s \cdot \lambda + \lambda \cdot s)\cdot v = v$, which means that $\lambda : F_\pm \to F_\mp$ is a null-homotopy for the identity. Hence it is locally contractible off $Z(s)$, and supported on $Z(s)$ by Lemma \ref{lem: coeqcon}. 
\end{proof}

Now we assume that $E$ is spin. Given a section $s\in \Gamma(E)$, we define the \emph{Clifford factorisation}
$$
\(\cS_E, \sqrt{-1}s\) \in \MF\(Y, -q(s,s)\)_{Z(s)}
$$
as in \eqref{eq: fact}. Tensoring with it defines a functor $\mf (Y,q(s,s))\to \mf (Y,0)_{Z(s)}$, which gives a homomorphism
\beq{S-E}
K\(\mf (Y,q(s,s))\)\To K_0\(Z(s)\).
\eeq

Even if $E$ is not spin, we can define the homomorphism \eqref{S-E} using a maximal isotropic subbundle $\Lambda\subset E$ -- we can use $\cS_\Lambda$ instead of $\cS_E$ and twist it by $[\surd{\det\Lambda}]$ after inverting $2$ in coefficients. In the last paragraph of Appendix \ref{AppA}, we prove that for two $\Lambda_1, \Lambda_2\subset E$, there is a line bundle whose square is $\det \Lambda_1 \otimes \det \Lambda_2^*$. It twists the factorisation $(\cS_{\Lambda_1}, \sqrt{-1} s)$ to give an isomorphism to $(\cS_{\Lambda_2}, \sqrt{-1} s)$. This proves the well-definedness of the homomorphism \eqref{S-E} in case there is no spin. 

Comparing the homomorphism \eqref{S-E} with the ones in \cite{OS,OT1} when $s$ is isotropic, one can prove its equivalence to the localised square root Euler class, $[(\cS_{E}, \sqrt{-1} s)]=\surd{\mathfrak{e}}(E,s)$ defined in \cite[Definition 5.7]{OT1}. It does not make sense to speak of the square of the localised square root Euler class although such a square exists in the non-localised case; see \cite[Proposition 5.4]{OT1}. The following Proposition may clarify how this can be understood even when $s$ is non-isotropic.

\begin{prop}\label{KSfan}
Applying the homomorphism \eqref{S-E} to the factorisation $(\cS_E,s)\in \MF (Y, q(s,s))$, we have
\begin{equation*}
\label{eq: Eulersqformula}
[(\cS_E, \sqrt{-1} s)\otimes (\cS_E, s)] = (-1)^n \mathfrak{e}(E, s)\ \in\ K^0\(Z(s), \ZZ\),  
\end{equation*}
where $\mathfrak{e}(E, s)$ is the localised $K$-theoretic Euler class in \cite{YP}.
\end{prop}

\begin{proof}
It is enough to prove that $(\cS_E, \sqrt{-1} s)\otimes (\cS_E, s)$ is isomorphic to the Koszul complex of $(E, s)$ shifted by $n$. We first explain what is going on after forgetting the section $s$ -- we have an isomorphism of modules between $\cS_E\otimes \cS_E$ and $Cl(E,q)[n]$ in this case. By Proposition \ref{prop: sqcorrection}, we have an isomorphism $\cS_E\otimes \cS_E\cong \cS_\Lambda \otimes \cS_{\Lambda^*}[n]$. On the other hand, the multiplicative morphism $Cl(E,q)\otimes \det \Lambda\to Cl(E,q)$ has the kernel $\langle \Lambda\rangle\otimes \det \Lambda$, inducing an injection $\cS_\Lambda\otimes \det L\into Cl(E,q)$. Hence we have a morphism $\cS_\Lambda \otimes \cS_{\Lambda^*}[n]\into Cl(E,q)^{\otimes 2}[n]$. Denoting by $(-)^t$ the operation transposing words in $Cl(E,q)$, we then obtain graded morphism, 
\begin{equation}
\label{eq: sqiso}
\cS_E \otimes \cS_E \rt{\sim} \cS_\Lambda \otimes \cS_{\Lambda^*}[n] \rt{(a, b) \mapsto a\cdot b^t} Cl(E, q)[n].
\end{equation}
It is surjective, and hence isomorphism by counting dimensions.

Now we would like to keep track of the differentials of the 2-periodic complex associated with the first module in \eqref{eq: sqiso}, in order to see that the induced 2-periodic complex structure on $Cl(E,q)[n]$ becomes the Koszul complex shifted by $n$. One can compute the 2-periodic complex $(\cS_E, \sqrt{-1} s)\otimes (\cS_E, s)$ is isomorphic to
\begin{equation*}
    \(
    \cS_E^+\otimes \cS_E^+ \oplus \cS_E^- \otimes \cS_E^- \xleftrightharpoons[D_+]{D_-} \cS_E^-\otimes \cS_E^+ \oplus \cS_E^+ \otimes \cS_E^-
    \), 
\end{equation*}
\begin{equation*}
    D_+ = 
    \begin{pmatrix}
    s \otimes 1 & -1\otimes s\\
    -1\otimes s & s \otimes 1
    \end{pmatrix}, \quad  
    D_- = 
    \begin{pmatrix}
    s \otimes 1 & 1\otimes s\\
    1\otimes s & s \otimes 1
    \end{pmatrix}.
\end{equation*}
To get this, we replaced one factorisation $(\cS_E, \sqrt{-1} s)$ by an isomorphic one
$$
\xymatrix{
\cS_E^+ \ar@/_1pc/[r]_-{-s} & \cS_E^- \ar@/_1pc/[l]_-{s}
}.
$$
We compute that $D$ corresponds to the below differential on $Cl(E,q)[n]$ under the identification \eqref{eq: sqiso},
$$
D(c) = (s\cdot c) + (-1)^{\vert c\vert+1} (c\cdot s^t), \quad c\in Cl(E,q),\  \vert c\vert = \textrm{word length of $c$}.
$$
With local basis of $E$, one can check that this $D$ is the interior product by the section $s$. Precisely, with a local orthonormal basis $\{e_1, \ldots, e_{2m}\}$ of $E$, we express $s$ and a local basis of $Cl(E, q)$ using $e_i$'s;
$$
s=\sum s_ie_i, \quad e_I = e_{i_1}\cdots e_{i_k}, \ I=\{i_1< \ldots < i_k\}\subset \{1, \ldots, 2m\}.
$$
Using these descriptions, we can compute
\begin{align*}
D(e_I) & = \sum_i\( s_i e_i \cdot e_I + (-1)^{\vert I \vert+1} s_ie_I \cdot e_i \) 
=\sum_{i_j \in I}(-1)^{j-1}s_{i_j} \iota_{e_{i_j}}(e_I)=\iota_s(e_I).
\end{align*}
Therefore, $(Cl(E, q), D)$ equals the (2-periodisation of) Koszul complex associated with $(E,s)$. 
\end{proof}

\subsection*{Isotropic reduction}
We slightly change our notation. Let $Z$ be a DM stack and $(E,q)$ be $SO(2n,\C)$ bundle over $Z$. Denoting by $p:E\to Z$ the projection, the tautological section $\tau_E$ of $p^*E$ cuts out the zero section $Z\subset E$. Now let $C\subset E$ be a cone over $Z$. Then the quadruple $(C,q=q(\tau_E,\tau_E)|_C, p^*E|_C,\tau_E|_C)$ plays a role of the previous $(Y,f,E,s)$ so that we have the homomorphism \eqref{S-E} constructed by using $(\cS_{p^*E\vert_C}, \sqrt{-1}\tau_E)$. We denote this factorisation or homomorphism by
\beq{Svir}
\cS^{\vir}_E\in \MF(C,-q)\quad\text{or}\quad K\(\mf(C, q)\) \To K_0\(Z,\ZZ\).
\eeq

Suppose there is an isotropic subbundle $K \subset E$ such that $K\subset C\subset K^\perp$. Then the reduced cone $C/K$ lies in the reduction $K^\perp/K$. Denoting by $\overline{q}$ the induced quadratic form on the reduction, we may understand it as a quadratic function on $C/K$. Its pullback by the quotient morphism $\pi: C\to C/K$ becomes $q$. So we have two homomorphisms
$$
\cS_{K^{\perp}/K}^{\vir},\ \ \pi^*(-)\otimes \cS_E^{\vir}: K(\mf (C/K,\overline{q})) \To K_0(Z,\ZZ).
$$

\begin{prop}
\label{prop: factisored}
The two operations above are equivalent up to twisting,
\begin{equation*}
\cS_{K^{\perp}/K}^{\vir} \otimes \sqrt{\det K}=\pi^*(-)\otimes \cS_E^{\vir}: K\(\mf(C/K, \overline{q})\)\To K_0\(Z,\ZZ\).
\end{equation*}
\end{prop}
\begin{proof}
Our strategy is to construct a family of $SO(2n,\C)$ bundles $\mathcal{E}$ over $Z\times \AA^1$ degenerating from $E$ to $\(K^\perp/K\)\oplus (K\oplus K^*)$, where the quadratic form on the right-hand side is $\overline{q}\oplus \mathrm{pair}_K$ with isotropic subbundle $K\subset \mathcal{E}$ degenerating form the given $K\subset E$ to 
$$
K\rt{0\oplus (1,0)} \(K^\perp/K\)\oplus (K\oplus K^*).
$$
We also construct a family of cones $\cC$ inside $\mathcal{E}$ degenerating from $C$ to $C/K\oplus K$ such that $K\subset \cC\subset K^\perp$.

We explain how we construct them later. Instead we explain how we can prove Proposition using them. The families defines a deformation of factorisations
$$
\cS_E^{\vir}\in \MF(C,-q)_Z\ \rightsquigarrow\ \cS_{(K^\perp/K)\oplus (K\oplus K^*)}^{\vir}\in \MF(C/K\oplus K,-\overline{q}\oplus 0)_Z.
$$ 
This induces, for $F_\bullet \in \mf (C/K,\overline{q})$, a deformation of $2$-periodic complexes
$$
\pi^*F_\bullet\otimes \cS_E^{\vir}\in \mf(C,0)_Z\ \rightsquigarrow\ \mathrm{pr}_1^*F_\bullet\otimes  \cS_{(K^\perp/K)\oplus (K\oplus K^*)}^{\vir}\in \mf(C/K\oplus K,0)_Z,
$$
which is flat over $\AA^1$. The $\AA^1$-homotopy invariance of $K$-groups of coherent sheaves tells us that the Gysin pullback to fibre $t\in \AA^1$ of its (signed) cohomology is independent of $t$. Since taking cohomology commutes with Gysin pullback to fibre \cite[Lemma 2.15]{OS}\footnote{\label{OS215}\cite[Lemma 2.15]{OS} applies to a $2$-periodic complex on a trivial family $Y\times \AA^1$, but the proof works well for a complex on any $\cY$ over $\AA^1$.}, the flatness guarantees that the cohomology of each fibre is independent. Hence, it is enough to prove that 
$$
\cS_{K^\perp/K}^{\vir}\otimes \sqrt{\det K}=\mathrm{pr}_1^*(-)\otimes \cS_{(K^\perp/K)\oplus (K\oplus K^*)}^{\vir},
$$
as homomorphisms from $K(\mf(C/K,\overline{q}))$ to $K_0(Z,\ZZ)$. Here, $\mathrm{pr}_1$ denotes the first projection $C/K\oplus K\to C/K$. Then the first upshot is that the factorisation is decomposed into
$$
\cS_{(K^\perp/K)\oplus (K\oplus K^*)}^{\vir}\cong \cS_{K^{\perp}/K}^{\vir}\boxtimes \cS_{K\oplus K^*}^{\vir}.
$$
The second upshot is the latter one $\cS_{K\oplus K^*}^{\vir}\in \MF(\mathrm{tot}K,0)$ is the Koszul factorisation $(\Lambda^\bullet K^*,\tau_K)\otimes \surd{\det K}=\mathfrak{e}(K,\tau_K)\otimes \surd{\det K}$; see Propositions \ref{lem: IandS}, \ref{prop: sqcorrection} and \eqref{eq: fact}, where $\mathfrak{e}$ denotes the $K$-theoretic Euler class and $\tau_K$ is the tautological section of $K$. Hence as homomorphisms from $K(\mf (C/K,\overline{q}))$ to itself, we have
$$
\mathrm{pr}_1^*(-)\otimes \cS_{K\oplus K^*}^{\vir}= \mathrm{pr}_1^*(-)\otimes \mathfrak{e}(K,\tau_K) \otimes  \sqrt{\det K}= \sqrt{\det K},
$$
providing the proof.

It remains to construct such families described above. The bundle $\mathcal{E}$ is taken to be the isotropic reduction of 
$$
K\rt{i\oplus (\lambda,0)}E\oplus (K\oplus K^*),
$$
where $i:K\into E$ is the given isotropic subbundle and $\lambda$ is the coordinate on $\AA^1$. Then the isotropic subbundle $K\subset \mathcal{E}$ is given by the quotient of
$$
K\rt{1\oplus (\lambda,0)} K\oplus (K\oplus 0)\cong K\oplus K.
$$
The family $\cC$ of cones is constructed as the quotient of
$$
K\rt{j\oplus (\lambda,0)} C\oplus (K\oplus 0)\cong C\oplus K,
$$
where $j:K\into C$ denotes the embedding.
\end{proof}

\section{Lagrangian class in $K$-theory}
\subsection{Normal form of the tangent complex $T_{L/U}$}\label{sect:normalform}
We fix a critical model $(U,f)$ of $M\cong \mathrm{crit}(f)$ and let $p:L\to M$ be a $(-1)$-Lagrangian space. We assume that
\begin{itemize}
\item $L$ is a quasi-compact Artin stack with affine diagonal, 
\item $L\to M$ is a DM morphism,
\item $L$ has the resolution property -- every coherent sheaf on $L$ is a quotient of a locally free sheaf of finite rank, and
\item $L\to U$ is oriented and has even virtual dimension. 
\end{itemize}
On $L$ we have the following exact triangle of tangent complexes
\beq{t}
T_{L/M}\To T_{L}\To p^*T_{M}.
\eeq
Since $L$ is $(-1)$-shifted Lagrangian, there is another exact triangle
\beq{l}
T_{L}^\vee[-2]\To T_{L}\To p^*T_{M}.
\eeq
Comparing \eqref{t} and \eqref{l}, we have $T_{L/M}\cong T_L^\vee[-2]$. The last term $p^*T_M$ in \eqref{l} is represented by the symmetric complex $\{d^2f: p^*(T_U|_M)\to p^*(\Omega_U|_M)\}$. Using this, we would like to find a {\em normal form} of the tangent complex $T_L$ of the $(-1)$-shifted Lagrangian $L$.

Consider a composition of complexes $T_{L}\to p^*T_{M}\to p^*(T_U|_M)$ whose cocone is quasi-isomorphic to $T_{L/U}$. Then from the diagram of exact triangles
$$
\xymatrix{
& p^*(\Omega_U|_M)[-1] \ar@{=}[r]\ar[d] & p^*(\Omega_U|_M)[-1]\ar[d] \\
T_{L} \ar@{=}[d]\ar[r] &p^*T_{M} \ar[d]\ar[r] & T_{L/M}[1]\cong T_L^\vee[-1] \ar[d]\\
T_{L} \ar[r] & p^*(T_U|_M)\ar[r] & T_{L/U}[1]\cong T_{L/U}^\vee[-1],\\
}
$$
we have a shifted symmetric quasi-isomorphism $T_{L/U}\cong T_{L/U}^\vee[-2]$. Applying the method of getting the self-dual forms \cite[Equations (48), (49)]{OT1} for the tangent complex of $(-2)$-symplectic space, we can find a symmetric representative.

\begin{prop}
On $L$, there is a bundle map $a:T\to E$, with $(E,q)$ an even rank $SO$-bundle such that the tangent complex $T_{L/U}$ is represented by
\beq{n2}
T_{L/U}\=\left\{T\rt{a} E\rt{a^*} T^*\right\}.
\eeq
\end{prop}
We call such a representative \eqref{n2} a {\em normal form of $T_{L/U}$}.

\begin{proof} The construction of a normal form follows the proof of \cite[Proposition 4.1]{OT1} using the resolution property of $L$. The resulting complex is a symmetric three-term complex since $L\to M$ is a DM morphism. The resulting bundle $E$ has even rank because $L\to U$ has even virtual dimension. Moreover, it is an $SO$-bundle since $L\to U$ is oriented.
\end{proof}

Fix a local smooth embedding $L\subset A$ which is smooth over $U$, with defining ideal $I$. Mimicking the proof in \cite[Proposition 4.1]{OT1}, the representative \eqref{n2} of $T_{L/U}$ can be taken to have a representative of the morphism $T^\vee_{L/U}\to \LL_{L/U}$ by a chain map
\beq{nwithL}
\xymatrix{
T_{L/U}^\vee \ar[d] & = & \{\ T\ar[r]^-a & E \ar[r]^-{a^*} \ar@{->>}[d]& T^* \} \ar@{->>}[d] \\
\LL_{L/U} & = & & I/I^2 \ar[r]^-d & \Omega_{A/U}|_L.
}
\eeq
It allows the Behrend-Fantechi cone of $L\to U$ to embed into $E$.

\subsection{$K$-theoretic Lagrangian class}\label{Sec:pullback}
In this and next Sections, we would like to construct a homomorphism
\beq{Lvir}
L^{\vir}_{(U,f)}: K\(\MF(U,f)\)\To K_0\(L,\ZZ\),
\eeq
called the {\em $K$-theoretic Lagrangian class}, or simply {\em Lagrangian class} for short. Then we prove that it is independent of choices of normal forms.

\smallskip
Let us choose a local smooth embedding $L\subset A$, smooth over $U$. By abusing notation, we denote by $f$ the pullback function on $A$ via $A\to U$. 

\begin{lem} \label{Lem:fq}
The pullback function $f$ on $A$ lies in $I_{L/A}^2$, where $I_{L/A}\subset \cO_A$ is the ideal sheaf of $L\subset A$. In particular $f$ lies in $I_{L/U}^2$.\footnote{H. Park pointed out that this lemma can be proved using the fact that $L\to U$ is $f$-locked $(-2)$-symplectic fibration \cite[Corollary 3.13]{Pa}, which avoids the difficulty that \cite[Example 3.6]{JS} is formulated only for schemes.}
\end{lem}
\begin{proof} 
A crucial fact we are going to use is that $L$ is a (shifted) Lagrangian of $\mathrm{Crit}(f)$. By \cite[Example 3.6]{JS}, this fact guarantees, locally, there is an ambient space $V\supset L$ with smooth map $p:V\to U$; and quadratic bundle $(E,q)$ on $V$ and section $s$ of $E$ satisfying $q(s,s)=p^*f$, $L\cong Z(s)$. In particular, $p^*f$ lies in $I_{L/V}^2$; see \eqref{-1local}. By multiplying $\C^{2r}$ if necessary, we may assume that $V$ contains an open set of $A$ as a closed subscheme. Letting $i:A\into V$ denote this embedding, we have a morphism $i^*I_{L/V}\to I_{L/A}$\vspace{-0.5mm} which gives $i^*I_{L/V}^2\to I_{L/A}^2$. Hence the function $i^*p^*f$ lies in $I_{L/A}^2$.
\end{proof}

Lemma \ref{Lem:fq} together with Theorem \ref{thm: spfunctor} shows that there exists a specialisation DG functor
$$
\mathrm{sp}: \MF (U,f)\To \mf (C_{L/U},\lambda^{-2}f),
$$
where $C_{L/U}$ is the intrinsic normal cone inside the bundle stack $[E/T]$. We denote by $\pi:C\to C_{L/U}$ the fibre product $C:=C_{L/U}\times_{[E/T]} E$, called the Behrend-Fantechi cone. Lemma \ref{Lem:fq} (or \cite[Corollary 3.1.3]{Pa}) shows that not only is $\lambda^{-2}f$ defined on $C_{L/U}$, but its pullback $\pi^*(\lambda^{-2}f)$ coincides with the quadratic function $q$ under the embedding $C\subset E$. Thus, we have a pullback functor
$$
\pi^*:\mf (C_{L/U},\lambda^{-2}f)\To \mf (C,q).
$$
The twisted composition $\sqrt{\det T^*}\otimes (\pi^*\circ \mathrm{sp})$ induces a homomorphism between $K$-groups, $K(\MF(U,f))\to K(\mf (C,q))$. Composing with the homomorphism $\cS_E^{\vir}$ \eqref{Svir}, we define the $K$-theoretic Lagrangian class \eqref{Lvir}
$$
L^{\vir}_{(U,f)}: K\(\MF(U,f)\)\To K_0\(L,\ZZ\).
$$

\subsection{Independence of the choice of normal form}
In this section we prove the independence. Remember the Lagrangian class $L_{(U,f)}^{\vir}$\vspace{-0.7mm} is defined to be the composition $\cS_E^{\vir}\otimes \surd{\det T^*}\otimes \pi^*\circ \mathrm{sp}$. In the composition, $\cS_E^{\vir}\otimes \surd{\det T^*}\otimes \pi^*$ depends on the choice of normal form \eqref{n2} while the specialisation $\mathrm{sp}$ is well-defined.

\begin{prop}\label{thm:indpullback}
The Lagrangian class $L_{(U,f)}^{\vir}$ is independent of choices of normal forms \eqref{n2}.
\end{prop}
\begin{proof}
We base our proof on \cite[Section 4.2]{OT1} in which it is proved that two normal forms are related by a combination of (1) isotropic reductions and (2) a one-parameter family. We say a normal form $A^\bullet=\{A\to B\to A^*\}$ is an isotropic reduction of another normal form $T^\bullet=\{T\to E\to T^*\}$ if there is a bundle $K$ satisfying
$$
A^\bullet\oplus \{K\to K\oplus K^*\to K^*\}\ \cong\ T^{\bullet}
$$
locally. Globally, such $K$ should be a subbundle of $T$ with the quotient isomorphic to $A$ which becomes an isotropic subbundle of $E$ via $T\to E$ whose reduction $K^{\perp}/K$ is isomorphic to $B$. So, if $A^\bullet$ is an isotropic reduction of $T^\bullet$, there are diagrams
\begin{equation}
\label{eq: isoredcompl}
\xymatrix@R=5mm{
T \ar[r] \ar@{=}[d] & (K^\perp) \ar[r] \ar@{_{(}->}[d]& A^*\ar@{_{(}->}[d] &&K  \ar@{_{(}->}[d]  \ar@{=}[r] & K \ar@{_{(}->} [d]&\\
T \ar[r]  & E \ar[r] \ar@{->>}[d]& T^* \ar@{->>}[d] & \textrm{and} & T \ar[r] \ar@{->>}[d]& (K^\perp) \ar@{->>}[d] \ar[r] & A^*\ar@{=}[d]\\
& K^* \ar@{=}[r]& K^*& &  A \ar[r] & B \ar[r] & A^*,\\
}
\end{equation}
(cf. the dual diagrams \cite[Equation (81)]{OT1}). A one-parameter family between two normal forms $A^\bullet_0=\{A\to B_0\to A^*\}$ and $A^\bullet_1=\{A\to B_1\to A^*\}$ is considered only when they have the same $A$ because we only meet this situation. It is a family of normal forms $A^\bullet_t=\{A\to B_t\to A^*\}$ over $L\times \AA^1$ (cf. \cite[Equation (77)]{OT1}).

We first prove that $\cS^{\vir}_E\otimes\sqrt{\det T^*}\otimes \pi^*$ is independent of choices of normal forms related by isotropic reductions. Suppose that the normal form $A^\bullet=\{A\to B\to A^*\}$ is an isotropic reduction of $T^\bullet=\{T\to E\to T^*\}$ with the diagrams \eqref{eq: isoredcompl}. Let us denote by $C_E\subset E$ and $C_B\subset B$ the corresponding Behrend-Fantechi cones, respectively. As discussed in \cite[Proposition 4.5]{OT1}, $C_E$ lies in $K^{\perp}\subset E$ and contains $K$. Its quotient $C_E/K\subset K^{\perp}/K$ maps to $C_B$ via the canonical isomorphism $K^{\perp}/K\cong B$ induced by the middle vertical morphism in the second diagram of \eqref{eq: isoredcompl}. With these facts and the projection maps $\pi_{E}: C_E \to C_{L/U}$ and $\pi_B:C_B \to C_{L/U}$ which intertwine with $\pi_K:C_E\to C_E/K\cong C_B$, we conclude from Proposition \ref{prop: factisored} that
\begin{align*}
\cS_E^\mathrm{vir}\otimes \sqrt{\det T^*}\otimes \pi^*_{E}(\mathrm{sp}(\Fd))   = &\cS_E^\mathrm{vir}\otimes \sqrt{\det T^*} \otimes (\pi^*_K \circ \pi^*_{B})(\mathrm{sp}(\Fd))  \\
= &\cS_{B}^\mathrm{vir} \otimes \sqrt{\det T^*} \otimes \sqrt{\det K} \otimes  \pi_{B}^*( \mathrm{sp} (\Fd)) \\
= & \cS_{B}^\mathrm{vir} \otimes \sqrt{\det A^*} \otimes \pi_{B}^*( \mathrm{sp} (\Fd)) .
\end{align*}

Next, suppose that two normal forms $A^\bullet_0=\{A\to B_0\to A^*\}$ and $A^\bullet_1=\{A\to B_1\to A^*\}$ are related by a family $A^\bullet_t=\{A\to B_t\to A^*\}$\footnote{When $L$ is a projective scheme, we can prove that two normal forms connected by a one-parameter family are isomorphic. Moreover, we can choose the isomorphism so that the Behrend-Fantechi cones are preserved. This provides an alternative proof of the invariance of $\cS^{\vir}_E\otimes\sqrt{\det T^*}\otimes \pi^*$ between these two normal forms when $L$ is a projective scheme.}. As in the paragraph containing \cite[Equation (77)]{OT1}, using the local cutout models of $L$ in the proof of Lemma \ref{Lem:fq}, the Behrend-Fantechi cones $C_0\subset B_0$ and $C_1\subset B_1$ are related by a family of cones on which the two functions $f/\lambda^2$ and $q_t$ are matched. Here, the function $q_t$ is induced by the quadratic form on $B_t$. Then the $\AA^1$-homotopy invariance of $K$-groups of coherent sheaves tells us that the Gysin pullback of
\beq{cSBt}
\cS_{B_t}^\mathrm{vir} \otimes \sqrt{\det A^*}\otimes  \pi_{B_t}^* (\mathrm{sp}(\Fd))  \in K_0\( L\times \mathbb A^1,\ZZ\)
\eeq
to the fibre $t\in \AA^1$ is independent of $t$. It is the same as the cohomology of Gysin pullback if we view \eqref{cSBt} as a factorisation; see \cite[Lemma 2.15]{OS} and Footnote \ref{OS215}. Since it is flat, the Gysin pullback is the ordinary pullback. Hence the pullbacks to $t=0$ and $t=1$ of \eqref{cSBt} are
$$
\cS_{B_0}^\mathrm{vir} \otimes \sqrt{\det A^*}\otimes  \pi_{B_0}^* (\mathrm{sp}(\Fd)) \ \ \text{and}\ \ \cS_{B_1}^\mathrm{vir} \otimes \sqrt{\det A^*}\otimes  \pi_{B_1}^* (\mathrm{sp}(\Fd)) 
$$
in $K_0( L,\ZZ)$, respectively. This proves the invariance of $\cS^{\vir}_E\otimes\sqrt{\det T^*}\otimes \pi^*$ under a one-parameter family of normal forms.
\end{proof}

\subsection{$(-2)$-shifted symplectic}
Let $L$ be a $(-2)$-symplectic DM stack. The morphism $L\to \Spec\C$ is then $(-1)$-Lagrangian of the point with critical model $(\Spec\C, 0)$. Hence we obtain the Lagrangian class
$$
L^{\vir}:=L^{\vir}_{(\Spec\C, 0)}: K\(\MF(\Spec\C,0)\)\cong \Z\ \To\ K_0\(L,\ZZ\).
$$

\begin{prop}\label{1toO} The Lagrangian class $L^{\vir}$ for a $(-2)$-symplectic DM stack $L$ takes $1$ to the virtual structure sheaf $\widehat{\cO}^{\vir}_L$ \cite[Definition 5.9]{OT1}.
\end{prop}
\begin{proof} By Example \eqref{ex2}, we have $\pi^*(\mathrm{sp}(\cO_U))=\cO_C$. Hence $L^{\vir}(\cO_U)=\surd{\det T^*}\otimes \cS^{\vir}_E$ once we fix the normal form $\{T\to E\to T^*\}$ of $T_L$. In the paragraph above Proposition \ref{KSfan}, we have discussed $\cS^{\vir}_E=\surd{\mathfrak{e}}(E|_C,\tau_E)$, which is $\surd{0}^{\;*}_E[\cO_C]$ \cite[Definition 5.8]{OT1}. Hence by \cite[Definition 5.9]{OT1}, we have $L^{\vir}(\cO_U)=\surd{\det T^*}\cdot \surd{0}^{\;*}_E[\cO_C]=\widehat{\cO}^{\vir}_L$.
\end{proof}

Finally we obtain our main theorem.

\begin{thm}\label{main2}
Let $L$ be an oriented $(-1)$-Lagrangian of $\mathrm{crit}(f)$. Suppose 
\begin{itemize}
\item $L$ is a DM stack and $L\to M$ is a DM morphism,
\item $L$ has the resolution property,
\item the relative virtual dimension of $L\to U$ is even. 
\end{itemize}
Then the Lagrangian class $L^{\vir}_{(U,f)}$\vspace{-0.7mm} \eqref{Lvir} is well-defined. If $L$ is $(-2)$-symplectic DM stack, then $L^{\vir}$ takes $\cO_{\Spec\C}$ to $\widehat{\cO}^{\vir}_L$. 
\end{thm}

\subsection{With global cutout data of $L$}
We assume that $L$ is isomorphic to the zero locus of the section $s$ in the diagram \eqref{-1local}. In this case, we have the canonical normal form of $T_{L/U}$,
$$
T_{L/U}\cong \left\{ T_V|_L \rt{ds} E|_L \rt{ds^*} T^*_V|_L \right\}.
$$
The Behrend-Fantechi cone $C\subset E|_L$ is the normal cone $C\cong C_{L/V}\subset E|_L$, where the embedding is obtained by the surjection $E^*\onto I_{L/V} (\subset \cO_V)$.

Applying Proposition \ref{Kcomm} to $p:V\to U$, we have the following commutativity of $K$-groups
$$
\xymatrix{
K\(\MF (U,f)\)\ar[r]^-{\mathrm{sp}}\ar[d]_-{p^*} & K\(\mf (C_{L/U},q)\) \ar[d]^-{p^*}\\
K\(\MF (V,p^*f)\)\ar[r]^-{\mathrm{sp}} & K\(\mf(C_{L/V},p^*q)\).
}
$$
Hence, with the global cutout model \eqref{-1local}, we obtain 
$$
L^{\vir}_{(U,f)}(\Fd)=\cS^{\vir}_{E|_L}\otimes \mathrm{sp}\( \sqrt{\det T^*_V|_L}\otimes p^*\Fd\)\ \in\ K_0\(L,\ZZ\).
$$
This tells us that the definition suggested in Introduction \eqref{p3} is matched with $L^{\vir}_{(U,f)}$ \eqref{Lvir}.

\section{Equivalent critical models}\label{Sect:twoc}
Given a critical model $(U,f)$ with $SO(2m,\C)$ bundle $(Q,q_Q)$, we associate to it the critical model $(\mathrm{tot}\;Q,f+q_Q)$ on the total space, which is equivalent to $(U,f)$ in the sense that their critical loci are isomorphic as $(-1)$-shifted symplectic spaces. Here, $f$ on $\mathrm{tot}\;Q$ denotes the pullback function via the projection map $\varphi:\mathrm{tot}\;Q\to U$ and $q_Q$ denotes the quadratic function induced by the quadratic form, by abusing notations. Let us define the factorisation
\beq{Kvir}
\cK_Q^{\vir}:=(\cS_{\varphi^*Q},\tau_Q)\in \MF(\mathrm{tot}\;Q, q_Q),
\eeq
using the Clifford module $\cS_{\varphi^*Q}$ of $Cl(\varphi^*Q,q_Q)$ and the tautological section $\tau_Q$ of $\varphi^*Q$ (cf. $\cS^{\vir}_Q\in \MF(\mathrm{tot}\;Q, -q_Q)$ in \eqref{Svir} was defined by using the same $S_{\varphi^*Q}$, but different section $\sqrt{-1}\tau_Q$). It is known as Kn\"orrer periodicity that $\cK_Q^{\vir}$ induces an equivalence of derived categories of matrix factorisations ~\cite{Te}
\begin{equation}
\label{eq: GenKnoerrer}
\varphi^*(-)\otimes \cK_Q^{\vir}[m] : \MF(U, f) \To \MF(\mathrm{tot}\;Q, f+q_Q).
\end{equation}
The goal of this section is to prove the following theorem.

\begin{thm}\label{thm:critind} 
Let $L\to \mathrm{crit}(f)$ be a $(-1)$-Lagrangian. For a $SO(2m,\C)$-bundle $(Q,q_Q)$ on $U$, the two Lagrangian classes obtained by critical models $(U,f)$ and $(\mathrm{tot}\;Q,f+q_Q)$ commute with the Kn\"orrer periodicity,
\beq{KLL}
\xymatrix{
K\(\MF(U, f)\) \ar[r]^-{L_{(U,f)}^{\vir}} \ar[d]_-{\eqref{eq: GenKnoerrer}}& K_0\(L,\ZZ\)\\
K\(\MF(\mathrm{tot}\;Q, f+q_Q)\). \ar[ur]_-{\ \ L_{(\mathrm{tot}\;Q,f+q_Q)}^{\vir}}&
}
\eeq
In particular, for a $(-2)$-symplectic $L$, the Lagrangian class
$$
L^{\vir}_{(\C^{2m},xy)}: K\(\MF(\C^{2m}, x_1y_1+\cdots +x_ry_r)\) \To K_0\(L\)
$$
takes $\cK_{\C^{2m}}^{\vir}[m]$ \eqref{Kvir} to the virtual structure sheaf $\widehat{\cO}_L\in K_0\(L,\ZZ\)$.
\end{thm}

\subsection{Commutativity of Lagrangian classes for equivalent critical models}
We first prove that the diagram \eqref{KLL} is commutative. We prove the commutativity of \eqref{KLL} using two key lemmas and postpone their proofs to the latter part of this section.

The first lemma is about normal forms of $T_{L/U}$ and $T_{L/\mathrm{tot}\;Q}$.

\begin{lem}\label{1stlemma} On $L$, we can find a normal form $\{T\xrightarrow{a} E\xrightarrow{a^*} T^*\}$ \eqref{n2} of $T_{L/U}$ whose sum with $Q[-1]$,
\beq{n3}
\{T\rt{(a,0)} E\oplus Q\rt{a^*+0} T^*\},
\eeq
becomes a normal form of $T_{L/\mathrm{tot}\;Q}$. Here, by abusing notation, $Q$ is the pullback to $L$ via $L\to U$. Moreover, the Behrend-Fantechi cone is also decomposed into $C\oplus Q\subset E\oplus Q$ accordingly.
\end{lem}

The quotient by $T$ of the Behrend-Fantechi cone $C\oplus Q$ in Lemma \ref{1stlemma} gives a natural identification of the intrinsic normal cones $C_{L/\mathrm{tot}\;Q}\cong C_{L/U}\oplus Q$. Using this, we state the second lemma which is the commutativity of the specialisations. 
\begin{lem}\label{2ndlemma}
Along the identification $C_{L/\mathrm{tot}\;Q}\cong C_{L/U}\oplus Q$ the function $f/\lambda^2\boxplus q_Q$ on the right-hand side pulls back to $(f+q_Q)/\lambda^2$. It defines the right vertical morphism in the diagram below making it commutative,
\beq{UQUQ}
\xymatrix{
K\(\MF (U,f)\) \ar[r]^-{\mathrm{sp}}\ar[d]_-{\eqref{eq: GenKnoerrer}} & K\(\mf(C_{L/U}, f/\lambda^2)\) \ar[d]^-{\eqref{eq: GenKnoerrer}} \\
K\(\MF(\mathrm{tot}\;Q, f+q_Q)\) \ar[r]^-{\mathrm{sp}} & K\(\mf(C_{L/\mathrm{tot}\;Q}, (f+q_Q)/\lambda^2)\).
}
\eeq
\end{lem}

\subsection*{Commutativity of \eqref{KLL}} Denoting by $\pi:C\to C_{L/U}$ the projection from the Behrend-Fantechi cone to the intrinsic normal cone, we recall that the Lagrangian class $L^{\vir}_{(U,f)}$ was defined to be $\cS^{\vir}_E\otimes \surd{\det T^*}\otimes \pi^*\circ \mathrm{sp}$. Hence, to prove the commutativity of \eqref{KLL} assuming Lemmas \ref{1stlemma} and \ref{2ndlemma}, we only need to prove that Lemma \ref{1stlemma} implies the commutativity of the diagram,
\beq{KKKdia1}
\xymatrix@C=25mm{
K\(\mf(C_{L/U}, f/\lambda^2)\) \ar[d]^-{\eqref{eq: GenKnoerrer}} \ar[r]^-{\cS^{\vir}_E\otimes\surd{\det T^*}\otimes\pi^*} & K_0\(L,\ZZ\) \\
K\(\mf(C_{L/\mathrm{tot}\;Q}, (f+q_Q)/\lambda^2)\). \ar[ru]_-{\ \ \ \ \ \ \ \ \ \ \ \ \ \cS^{\vir}_{E\oplus Q}\otimes\surd{\det (T\oplus Q)^*}\otimes\pi^*}
}
\eeq
The commutativity of tensor product by a matrix factorisation and smooth pullback for factorisations comes from the one for sheaves because we don't have to derive those functors. By the same reason, two smooth pullbacks are also commutative. This proves $\eqref{eq: GenKnoerrer}\circ \pi^*=\pi^*\circ \eqref{eq: GenKnoerrer}$. So the commutativity of \eqref{KKKdia1} comes from that of the following,
\beq{KKKdia2}
\xymatrix@C=25mm{
K\(\mf(C, q)\) \ar[d]^-{\eqref{eq: GenKnoerrer}} \ar[r]^-{\cS^{\vir}_E\otimes\surd{\det T^*}} & K_0\(L,\ZZ\) \\
K\(\mf(C\oplus Q, q\boxplus q_Q)\). \ar[ru]_-{\ \ \ \ \ \ \ \ \ \ \ \ \ \cS^{\vir}_{E\oplus Q}\otimes\surd{\det (T\oplus Q)^*}}
}
\eeq
Recall that the Kn\"orrer periodicity \eqref{eq: GenKnoerrer} is $\varphi^*(-)\otimes \cK^{\vir}_Q[m]$. In the upward diagonal arrow, we have $\cS^{\vir}_{E\oplus Q}\cong \cS^{\vir}_E\boxplus \cS^{\vir}_Q$ and $\surd{\det Q}\cong \cO_{C\oplus Q}$. Using these, it is enough to prove that $\varphi^*(-)\otimes \cK^{\vir}_Q\otimes \cS^{\vir}_Q[m]: K(\mf(C,q))\to K(\mf(C\oplus Q,q))$ equals the pushforward by the inclusion $i:C\into C\oplus Q$ that induces the following diagram,
$$
\xymatrix@C=30mm{
K\(\mf(C, q)\) \ar[d]^-{\eqref{eq: GenKnoerrer}} \ar[r]^-{\cS^{\vir}_E\otimes\surd{\det T^*}} & K\(\mf(C,0)_L\) \ar[d]^-{i_*}  \\
K\(\mf(C\oplus Q, q\boxplus q_Q)\) \ar[r]_-{\cS^{\vir}_{E\oplus Q}\otimes\surd{\det (T\oplus Q)^*}} & K\(\mf(C\oplus Q,0)_L\).
}
$$ 
By Proposition \ref{KSfan}, $\cK^{\vir}_Q\otimes \cS^{\vir}_Q[m]$ is the Koszul complex of $Q$ with the tautological section which is isomorphic to the (totalisation of) structure sheaf $i_*\cO_C\in \mf(C\oplus Q,0)_C$. Hence the projection formula tells us
$$
\varphi^*(-)\otimes \cK^{\vir}_Q\otimes \cS^{\vir}_Q[m]\cong \varphi^*(-)\otimes i_*(\cO_C)\cong i_*(i^*\varphi^*(-))\cong i_*(-).
$$
Then taking cohomology proves the commutativity of \eqref{KKKdia2}.

\subsection*{Proof of Lemma \ref{1stlemma}} In the derived category $D^b(L)$, we have the following diagram of exact triangles
$$
\xymatrix{
T_{U/\tot Q}|_L[-1] \ar[r]\ar@{=}[d]& T_U|_L[-1] \ar[r]\ar[d]& T_{\tot Q}|_L[-1] \ar[d]&  \\
T_{U/\tot Q}|_L[-1] \ar[r]& T_{L/U}\ar[r]\ar[d] & T_{L/\mathrm{tot} Q} \ar[d] \\
& T_L \ar@{=}[r]& T_L.
}
$$
Since the zero section morphism $U\into \tot\;Q$ induces a splitting exact sequence $0\to T_U\to T_{\tot Q}|_U\to T_{U/\tot Q}[1](\cong Q)\to 0$ of vector bundles, the morphism $T_{U/\tot Q}|_L[-1]\to T_U|_L[-1]$ in the upper horizontal exact triangle is zero. Hence the morphism $T_{U/\tot Q}|_L[-1]\to T_{L/U}$ in the middle horizontal exact triangle is also zero. Since $T_{U/\tot Q}|_L[-1]\cong Q[-2]$, for any normal form $\{T\to E\to T^*\}$ \eqref{n2} of $T_{L/U}$, we obtain a normal form $\{T\to E\oplus Q\to T^*\}$ \eqref{n3} of $T_{L/\tot Q}$.

Suppose the normal form comes equipped with the representative \eqref{nwithL} of $T_{L/U}^\vee\to \LL_{L/U}$. Since $\LL_{L/\mathrm{tot}\;Q}$ is the mapping cone of the zero morphism of $T_{U/\tot Q}|_L[-1]$ to $\LL_{L/U}$, the representative \eqref{nwithL} gives rise to a representative 
$$
\xymatrix{
T_{L/\tot Q}^\vee \ar[d] & = & \{\ T\ar[r]^-a & E\oplus Q \ar[r]^-{a^*} \ar@{->>}[d]^-{b\oplus 1}& T^* \} \ar@{->>}[d] \\
\LL_{L/\tot Q} & = & & I/I^2\oplus Q \ar[r]^-d & \Omega_{A/U}|_L.
}
$$
of $T_{L/\tot Q}^\vee\to \LL_{L/\tot Q}$, where $b:E\to I/I^2$ is the morphism given in \eqref{nwithL}. This proves that the Behrend-Fantechi cone from the normal form \eqref{n3} of $T_{L/\mathrm{tot}\;Q}$ is decomposed into $C\oplus Q\subset E\oplus Q$, where $C\subset E$ is that from \eqref{n2} of $T_{L/U}$.

\subsection*{Proof of Lemma \ref{2ndlemma}}
We first prove the first part of Lemma \ref{2ndlemma}: the function $f/\lambda^2\boxplus q_Q$ pulls back to $(f+q_Q)/\lambda^2$ via the identification $C_{L/\mathrm{tot}\;Q}\cong C_{L/U}\oplus Q$. We can check it locally assuming $Q$ to be a trivial bundle $U\times \AA^{2m}$. Then we have canonical isomorphisms
$$
C_{L/\mathrm{tot}\;Q} \cong C_{L/U}\times N_{\{0\}/\AA^{2m}}\quad\text{and}\quad M^o_{L/\mathrm{tot}\;Q} \cong M^o_{L/U}\times_{\AA^1} M^o_{\{0\}/\AA^{2m}}.
$$
We set $x_i,y_i$, $i=1,...,m$, be the coordinates on $\AA^{2m}$ with $q_Q=\sum_i x_iy_i$, and $X_i=\partial_{x_i},Y_i=\partial_{y_i}$, $i=1,...,m$, be the basis for $N_{\{0\}/\AA^{2m}}$. Then the function $q_Q/\lambda^2$ on $M^o_{\{0\}/\AA^{2m}}$ restricts to $q_Q=\sum_iX_iY_i$ on $N_{\{0\}/\AA^{2m}}$ if we embed $M^o_{\{0\}/\AA^{2m}}$ into $(\AA^{2m}\times N_{\{0\}/\AA^{2m}})\times \AA^1_{\lambda}$ as the zero loci of $(\lambda X_i- x_i,\lambda Y_i- y_i)$. 

We next prove the commutativity of the diagram \eqref{UQUQ}. As a preparation, we observe that the birational morphism $M^o_{\cK^{\vir}_Q}\to M^o_{U/\tot Q}$\vspace{-1.5mm} \eqref{MFd}, associated with the factorisation $\cK^{\vir}_Q$\vspace{-0.2mm} \eqref{Kvir} is in fact the identity. Indeed, the structure maps of $\cK^{\vir}_Q$ are given by the tautological section of $Q$, and hence divisions by $\lambda$ are well-defined on $M^o_{U/\tot Q}$\vspace{-0.3mm}. Hence, no blowup occurs. In general, by the same reason, for any $\Fd\in \MF(U,f)$, the birational morphism $M^o_{\varphi^*\Fd\otimes \cK^{\vir}_Q}\to M^o_{L/\tot Q}$ is the base-change of $M^o_{\Fd}\to M^o_{L/U}$, 
hence 
$$
\xymatrix{
C_{\varphi^*\Fd\otimes \cK^{\vir}_Q}\cong C_{\Fd}\oplus Q \ar[r]^-{\rho}\ar[d] & C_{L/\!\tot Q}\cong C_{L/U}\oplus Q \ar[d] \\
C_{\Fd} \ar[r]^-{\rho}& C_{L/U}
}
$$
is a fibre diagram. Using this, we would like to divide the diagram \eqref{UQUQ} into two parts
\beq{UQUQ2}
\xymatrix{
K\(\MF (U,f)\) \ar[r]\ar[d]_-{\eqref{eq: GenKnoerrer}} & K\(\MF(C_{\Fd}, f/\lambda^2)\) \ar[d]^-{\eqref{eq: GenKnoerrer}} \\
K\(\MF(\mathrm{tot}\;Q, f+q_Q)\) \ar[r] & K\(\MF(C_{\varphi^*\Fd\otimes \cK^{\vir}_Q}, (f+q_Q)/\lambda^2)\) 
}
\eeq
and 
\beq{UQUQ3}
\xymatrix{
K\(\MF(C_{\Fd}, f/\lambda^2)\) \ar[d]^-{\eqref{eq: GenKnoerrer}} \ar[r]^-{\rho_*} & K\(\mf(C_{L/U}, f/\lambda^2)\) \ar[d]^-{\eqref{eq: GenKnoerrer}} \\
K\(\MF(C_{\varphi^*\Fd\otimes \cK^{\vir}_Q}, (f+q_Q)/\lambda^2)\) \ar[r]^-{\rho_*} & K\(\mf(C_{L/\mathrm{tot}\;Q}, (f+q_Q)/\lambda^2)\),
}
\eeq
and prove the commutativity of each. The commutativity of \eqref{UQUQ3} is rather easier. For $G_\bullet \in 
\MF(C_{\Fd},f/\lambda^2)$, we have
$$
\rho_*\(\varphi^*G_\bullet \otimes \rho^*\cK^{\vir}_Q[m]\)=\rho_*\(\varphi^*G_\bullet\)\otimes \cK^{\vir}_Q[m]=\varphi^*\(\rho_*G_\bullet\)\otimes \cK^{\vir}_Q[m]
$$ 
by the projection formula \cite[Proposition 2.2.10]{CFGKS} and the smooth base-change \cite[Proposition 5.9]{BDFIK}.

It remains to show the commutativity of \eqref{UQUQ2}. In the rest of the proof, we denote $\cK_\bullet = \cK_Q^{\vir}[m]$ to simplify cluttering notations. Letting $\xi_\bullet$ and $\widetilde{\xi}_\bullet$ be the tautological factorisations on $C_{\Fd}$ and $C_{\varphi^*\Fd \otimes \cK_\bullet}$, respectively, for the factorisations $\Fd \in \MF(U,f)$ and $\varphi^*\Fd \otimes \cK_\bullet\in \MF(\tot\;Q, f + q_Q)$, showing the commutativity is equivalent to prove that
\beq{phixiK}
\varphi^*\xi_\bullet\otimes \cK_\bullet \cong \widetilde{\xi}_\bullet.
\eeq
Our plan is to show this on the deformation $M_{\mathrm{big}}:=M^o_{\varphi^*\Fd\otimes \cK^{\vir}_Q}\cong \varphi^*M_{\Fd}$\vspace{-1.0mm}. 

The factorisations define two Grassmannian bundles 
\begin{align*}
&G := Gr(r_+, F_+\oplus F_-)\times Gr(r_-, F_+\oplus F_-) \quad\text{and}\\
\widetilde G := Gr\(r, &(\Fd \otimes \cK_\bullet)_+\oplus (\Fd \otimes \cK_\bullet)_- \)\times Gr\(r, (\Fd \otimes \cK_\bullet)_+\oplus (\Fd \otimes \cK_\bullet)_-  \) ,
\end{align*}
over $M^o_{L/U}$ and $M^o_{L/\mathrm{tot}\;Q}$\vspace{-0.8mm}, respectively. The ranks of $(\Fd \otimes \cK_\bullet)_\pm$ are the same which we denote by $r=2^{m-1} (r_++r_-)$. Then $M_{\mathrm{big}}$ lies in $\widetilde{G}$. Considering $G$ to be its pullback via $\rho: M^o_{L/\tot Q}\to M^o_{L/U}$, $M_{\mathrm{big}}$ also lies in $G$ since $M_{\mathrm{big}}$ is the base-change of $M_{\Fd}$. Since the tautological factorisation on the former is $\wt{\xi}_\bullet$ and that on the latter is $\varphi^*\xi_\bullet$, it is sufficient to construct an embedding $G\into \wt{G}$ taking $M_{\mathrm{big}}$ to itself and an isomorphism \eqref{phixiK} using the tautological factorisations.

We explain the embedding $G\into \wt{G}$ in two steps. First, we consider a pair of subbundles $(A,B)\subset (F_+\oplus F_-) \oplus(F_+\oplus F_-)$ of rank $(r_+,r_-)$ representing an element of $G$. We send it to the following pair,
\begin{equation*}
\label{ABABFSFS}
\(A\otimes \cK_+ \oplus B \otimes \cK_-, A\otimes \cK_-\oplus B\otimes \cK_+\) \subset \(\Fd \otimes \cK_\bullet\) \oplus \(\Fd \otimes \cK_\bullet\),
\end{equation*}
where $\Fd \otimes \cK_\bullet$ denotes the sum $(\Fd \otimes \cK_\bullet)_+\oplus (\Fd \otimes \cK_\bullet)_-$ ignoring the grading. Each factor is a rank $r$ subbundle because $A\otimes \cK_+$ intersects $B \otimes \cK_-$  trivially, and so does $A\otimes \cK_-$ to $B\otimes \cK_+$. Next, we send each of the resulting pair under two different (even-degree) automorphisms $\Psi_1, \Psi_2 : \Fd \otimes \cK_\bullet \to \Fd \otimes \cK_\bullet$, respectively, defined below: 
\begin{align*}
	\Psi_1: &\((F_+ \otimes \cK_+) \oplus (F_- \otimes \cK_+) \oplus (F_+ \otimes \cK_-) \oplus (F_- \otimes \cK_-) \) \\
	&\To \((F_+ \otimes \cK_+) \oplus (F_- \otimes \cK_-) \oplus (F_+ \otimes \cK_-) \oplus (F_- \otimes \cK_+) \),
\end{align*}
$$
\Psi_1=
\left(
\begin{array}{cccc}
\id_{F_+}\otimes \id_{\cK_+}  &0 &0 &0  \\
0 &0&0 &\id_{F_-}\otimes \id_{\cK_-}  \\
\id_{F_+} \otimes \frac{\tau \cdot}{\lambda} &0 &\id_{F_+}\otimes \id_{\cK_-} & 0\\
0 & \id_{F_+}\otimes \id_{\cK_-}  & 0 & -\id_{F_-} \otimes \frac{\tau \cdot}{\lambda}
\end{array}
\right),
$$
and
\begin{align*}
	\Psi_2: &\((F_+ \otimes \cK_-) \oplus (F_- \otimes \cK_-) \oplus (F_+ \otimes \cK_+) \oplus (F_- \otimes \cK_+) \) \\
	&\To \((F_+ \otimes \cK_+) \oplus (F_- \otimes \cK_-) \oplus (F_+ \otimes \cK_-) \oplus (F_- \otimes \cK_+) \)
\end{align*}
$$
\Psi_2=
\left(
\begin{array}{cccc}
\id_{F_+} \otimes \frac{\tau \cdot}{\lambda} &0 &\id_{F_+}\otimes \id_{\cK_+}  &0  \\
0 &\id_{F_-}\otimes \id_{\cK_-} &0 & -\id_{F_-} \otimes \frac{\tau \cdot}{\lambda}\\
\id_{F_+}\otimes \id_{\cK_-} &0 &0& 0\\
0 &0 & 0 & \id_{F_+}\otimes \id_{\cK_-}
\end{array}
\right).
$$
In conclusion, we define $G \into \widetilde G$ to be the composition of these two,
$$
(A, B) \mapsto \left(\Psi_1(A\otimes \cK_+ \oplus B \otimes \cK_-), \Psi_2(A\otimes \cK_-\oplus B\otimes \cK_+)\right).
$$
If $(A,B)\in G$ is taken as even and odd parts of a subfactorisation $A_\bullet$ of
$$
\[F_+\oplus F_- \xrightleftharpoons[1 \oplus q]{q\oplus 1} F_+ \oplus F_-\]\in \MF(G,q),
$$
then it maps to a subfactorisation $\varphi^*A_\bullet\otimes (\cS_{\varphi^*Q},\tau_Q/\lambda)[m]$ of
$$
\[(F_+ \otimes \cK_+) \oplus (F_- \otimes \cK_-) \xrightleftharpoons[1 \oplus (q+q_Q)]{(q+ q_Q)\oplus 1} (F_+ \otimes \cK_-) \oplus (F_- \otimes \cK_+) \]\in \MF(\wt{G},q+q_Q).
$$
This embedding satisfies the required properties because it sends the graph $(\Gamma_{d_+/\lambda}, \Gamma_{d_-/\lambda})$ given by the structure maps of $\Fd$ to ones $(\Gamma_{\delta_+/\lambda}, \Gamma_{\delta_-/\lambda})$ given by the structure maps of $\rho^*\Fd \otimes \cK_\bullet$, in $\lambda\neq 0$. More precisely, $\delta_\pm$ are
$$
\delta_+= \begin{pmatrix}\id \otimes \tau\cdot & d_- \otimes \id \\ d_+\otimes \id & -\id \otimes \tau \cdot \end{pmatrix},\ \
\delta_-=  \begin{pmatrix}\id \otimes \tau\cdot & d_- \otimes \id \\ d_+\otimes \id & -\id \otimes \tau \cdot \end{pmatrix}. 
$$
Hence it takes $M_{\mathrm{big}}$ to itself and we get the induced isomorphism \eqref{phixiK}.

\subsection{$(-2)$-shifted symplectic} 
Let $L$ be a $(-2)$-symplectic DM stack. Since the critical models
$$
(\Spec\C,0)\quad\text{and}\quad (\AA^{2m}, x_1y_1+\cdots +x_my_m)
$$
are equivalent, that is, their critical loci are isomorphic as $(-1)$-symplectic spaces, the morphism
$$
L \to \mathrm{Crit}(\AA^{2m}, x_1y_1+\cdots +x_my_m)
$$
is $(-1)$-Lagrangian. These two critical models applied to \eqref{KLL} gives rise to the commutative diagram
$$
\xymatrix{
K(\MF(\Spec\C, 0)) \ar[r]^-{L^{\vir}} \ar[d]_-{\eqref{eq: GenKnoerrer}}& K_0\(L,\ZZ\)\\
K(\MF(\AA^{2m}, x_1y_1+\cdots +x_my_m). \ar[ur]_-{\ \ L_{(\AA^{2m},\sum x_iy_i)}^{\vir}}&
}
$$
By Proposition \ref{1toO}, $L^{\vir}(\cO_{\Spec\C})=\widehat{\cO}_L^{\vir}$. Since \eqref{eq: GenKnoerrer} takes $\cO_{\Spec\C}$ to $\cK_{\AA^{2m}}^{\vir}[m]$, this proves $L^{\vir}_{(\AA^{2m},\sum x_iy_i)}(\cK_{\AA^{2m}}^{\vir}[m])=\widehat{\cO}_L^{\vir}$, and Theorem \ref{thm:critind}.

\section{Torus localisation}\label{sect:torusl}
When $\T:=\C^*$ acts on the spaces $L$ and $U$ given as in Section \ref{sect:normalform}, we prove the localisation formula under the following assumptions;
\begin{itemize}
\item $L$ is a scheme,
\item $f: U \to \mathbb A^1$ is $\T$-invariant, 
\item $L \to U$ is $\T$-equivariant, and
\item the symplectic structure $T_{L/U}\rt{\sim} T_{L/U}^\vee[-2]$ is $\T$-equivariant.
\end{itemize}

Under the assumptions, denoting by $\MF ^{\T}(U,f)$ the category of $\T$-equivariant matrix factorisations, we can define $\T$-equivalent Lagrangian class
\beq{LTvir}
L^{\T}_{(U,f)}: K\(\MF ^{\T} (U,f)\)\To K^{\T}_0 (L)_{\loc},
\eeq
where $K_0(L)_{\loc}$ is the localised $K$-group defined in \cite[Page 54]{OT1}. Although one can check that the constructions of $\mathrm{sp}$ in Section \ref{Sect:sp} and $\cS^{\vir}_E$ \eqref{Svir} (and others) for the Lagrangian class $L^{\vir}_{(U,f)}$ are all $\T$-equivariant, it is conceptually simpler to view this as the Lagrangian class for the Artin stacks $[L/\T]\to [U/\T]$. 

To define the Lagrangian class 
\beq{LTTvir}
L_{(U^{\T}, f^{\T})}
: K\(\MF^{\T} (U^{\T}, f^{\T})\)\To K_0^{\T}\(L^{\T}\)_{\loc}
\eeq
for $L^{\T}\to \mathrm{crit}(U^{\T}, f^{\T}:=f|_{U^{\T}})$ without assuming it to be $(-1)$-Lagrangian, as in \eqref{Lvir}, we need to check if
\begin{itemize}
\item $f^{\T}$ lies in $I_{L^{\T}/U^{\T}}^2$ to have $\mathrm{sp}_{\T}:\MF (U^{\T},f^{\T})\to \mf (C_{L^{\T}/U^{\T}}, f^{\T}/\lambda^2)$,
\item the tangent complex $T_{L^{\T}/U^{\T}}$ is isomorphic to its dual shifted by $-2$ to have the Clifford factorisation on the Behrend-Fantechi cone with the function $q^{\T}$ induced by the quadratic form on $T_{L^{\T}/U^{\T}}$,
\item the function $f^{\T}/\lambda^2$ and $q^{\T}$ are the same on the Behrend-Fantechi cone.
\end{itemize}

Suppose that there is a local $\T$-equivariant embedding $L\into A$ with $A$ smooth over $U$. Since the fixed loci $L^{\T}$ is the fibre product $L\times_{A} A^{\T}$, the pullback of the defining ideal $I_{L/A}\subset \cO_A$ to $A^{\T}$ becomes $I_{L^{\T}/A^{\T}}$\vspace{-0.5mm}. Since $f$ lies in $I_{L/A}^2$\vspace{-0.5mm}, the restriction $f^{\T}$ also lies in $I_{L^{\T}/A^{\T}}^2$\vspace{-0.5mm}. By Footnote \ref{footnote3}, this proves $f^{\T}$ lies in $I_{L^{\T}/U^{\T}}^2$. 

Next, we claim that the fixed part of the restricted tangent complex $T_{L/U}|_{L^{\T}}$ is isomorphic to $T_{L^{\T}/U^{\T}}$. In fact, $T_{L/U}$ is the cocone of $T_L\to p^*T_U$, where $p$ denotes the morphism $L\to U$. Its fixed part on $L^{\T}$ is the cocone of $T_{L}|_{L^{\T}}^{\mathrm{fix}}\to p^*T_{U}|_{L^{\T}}^{\mathrm{fix}}$. It is proven in \cite{GP} that $T_{L}|_{L^{\T}}^{\mathrm{fix}}$ is isomorphic to the tangent complex $T_{L^{\T}}$ of the fixed loci $L^{\T}$. Since $U$ is smooth, $p^*T_{U}|_{L^{\T}}^{\mathrm{fix}}$ is isomorphic to the tangent bundle $p_{\T}^*T_{U^{\T}}$, where $p_{\T}:L^{\T}\to U^{\T}$ is the induced morphism by $p$. Hence the fixed part is the cocone of $T_{L^{\T}}\to p_{\T}^*T_{U^{\T}}$, which is isomorphic to $T_{L^{\T}/U^{\T}}$. This proves the claim, and hence we obtain the induced symplectic structure $T_{L^{\T}/U^{\T}}\cong T_{L^{\T}/U^{\T}}^\vee[-2]$.

Again with the local $\T$-equivariant embedding $L\into A$ with $A$ smooth over $U$, the deformation to the normal cone $M^o_{L^{\T}/A^{\T}}$\vspace{-0.9mm} lies in $M^o_{L/A}$ because $L^{\T}$ is the fibre product $L\times_{A} A^{\T}$. When we identify the central fibre $C_{L/A}$ with the Behrend-Factechi cone $C\subset E$ locally, the central fibre $C_{L^{\T}/A^{\T}}$\vspace{-0.7mm} of the former is the intersection with the fixed part $C|_{L^{\T}}\cap E|_{L^{\T}}^{\mathrm{fix}}$\vspace{-0.7mm}. Hence, the induced function on $C_{L^{\T}/A^{\T}}$\vspace{-0.5mm} is $f^{\T}/\lambda^2$ if we view it as the central fibre of the deformation; while it is the restricted quadratic function $q^{\T}:=q|_{C_{L^{\T}/A^{\T}}}$\vspace{-0.5mm} if we view it as the intersection. The parallel arguments in Propositions \ref{thm:indpullback} and \ref{1toO} show that $L_{(U^{\T},f^{\T})}$ defined in \eqref{LTTvir} is independent of the choice of normal form and sends $1$ to $\widehat{\cO}^{\vir}_{L^{\T}}$ when $L$ is $(-2)$-symplectic. 

\smallskip
With the Lagrangian classes \eqref{LTvir} and \eqref{LTTvir}; and embeddings $i_U:U^{\T}\into U$ and $i_L:L^{\T}\into L$ of the fixed loci, we would like to prove the following localisation theorem.

\begin{thm}\label{main4}
Suppose that the $(-1)$-Lagrangian $L\to \mathrm{crit}(f)$ is acted on by $\T$ satisfying the assumptions in the beginning of the section. Then the following diagram is commutative,
\beq{localcomm}
\xymatrixcolsep{11pc}
\xymatrix{
K\(\MF^{\mathsf{T}}(U, f)\) \ar[r]^-{L^{\T}_{(U,f)}} \ar[d]_-{i_U^*} & K_0^{\mathsf{T}}(L)_\mathrm{loc} \\
K\(\MF^{\mathsf{T}}(U^{\mathsf{T}}, f^{\mathsf{T}})) \ar[r]^-{\frac{L_{(U^{\T}, f^{\T})}}{\mathfrak{e}_{\T}(N_{U^{\T}/U})\cdot \sqrt{\mathfrak {e}_{\mathsf{T}}}(T_{L/U}^\mathrm{mov})}}  & K_0^{\mathsf{T}}\(L^{\mathsf{T}}\)_\mathrm{loc}, \ar[u]_-{i_{L*}}
}
\eeq
where $\mathrm{e}_{\T}$ and $\surd\mathfrak{e}_{\mathsf{T}}$ denote the equivariant $K$-theoretic Euler class and square root Euler class defined in \cite[Page 56]{OT1}, respectively; and $T_{L/U}^{\mathrm{mov}}$ is the moving part of $T_{L/U}|_{L^{\T}}$. In particular when $(U,f)=(\Spec\C,0)$, this reproves the ordinary localisation theorem \cite[Theorem 7.3]{OT1}.
\end{thm}

\subsection*{Why do we need $L$ scheme} The only place we have to use the assumption that $L$ is a scheme is the localisation theorem;
$$
i_{L*}:K_0^{\T}(L^{\T})_{\loc}\ \rt{\sim}\ K_0^{\T}(L)_{\loc}
$$
which is true when $L$ is a scheme \cite[Theorem 3.3(a)]{EG:Lo}. In this case, we also have the isomorphism
$$
0_*:K_0^{\T}(L^{\T})_{\loc}\ \rt{\sim}\ K_0^{\T}(C_{L^{\T}/L})_{\loc}
$$
induced by the zero section of the normal cone $C_{L^{\T}/T}$, however $L$ need not be a scheme for this. We can use a bundle containing $C_{L^{\T}/L}$ to show $0_*$ is an isomorphism. Once we have a condition on $L$ that guarantees $i_{L*}$ to be an isomorphism, we have the following preparation for the proof of the localisation theorem, Theorem \ref{main4}.

The two isomorphisms are connected by the ordinary specialisation map
$$
\xymatrix@C=2mm{
K_0^{\T}(L)_{\loc} \ar[rr]^-{\mathfrak{sp}} && K_0^{\T}(C_{L^{\T}/L})_{\loc} \\
& K_0^{\T}(L^{\T})_{\loc}. \ar[ru]^-{0_*} \ar[lu]_-{i_{L*}}&
}
$$ 
Note that $\mathfrak{sp}$ is defined using the deformation to the normal cone $M^o_{L^{\T}/L}$\vspace{-0.5mm}. Indeed, for an element $F\in K_0^{\T}(M^o_{L^{\T}/L})_{\loc}$\vspace{-0.7mm}, the Gysin pullback to $1\in \AA^1$ is mapped by $\mathfrak{sp}$ to the pullback to $0\in \AA^1$
$$
\mathfrak{sp}\(1^!F\)\= 0^!F.
$$
Hence, a vector bundle $G$ on $L$ is mapped by $\mathfrak{sp}$ to its pullback to $C_{L^{\T}/L}$ because $L$ pulls back to a bundle on $M^o_{L^{\T}/L}$\vspace{-0.3mm} whose Gysin pullbacks give such a relation.  

The key point is that the specialisation $\mathfrak{sp}$ is also an isomorphism. So instead of showing the commutativity of \eqref{localcomm}, we will prove that of
\beq{localcomm2}
\xymatrixcolsep{11pc}
\xymatrix{
K\(\MF^{\mathsf{T}}(U, f)\) \ar[r]^-{\mathfrak{sp}\circ L^{\T}_{(U,f)}} \ar[d]_-{i_U^*} & K_0^{\mathsf{T}}(C_{L^{\T}/L})_\mathrm{loc} \\
K\(\MF^{\mathsf{T}}(U^{\mathsf{T}}, f^{\mathsf{T}})) \ar[r]^-{\frac{L_{(U^{\T}, f^{\T})}}{\mathfrak{e}_{\T}(N_{U^{\T}/U})\cdot\sqrt{\mathfrak {e}_{\mathsf{T}}}(T_{L/U}^\mathrm{mov})}}  & K_0^{\mathsf{T}}\(L^{\mathsf{T}}\)_\mathrm{loc}. \ar[u]_-{0_{*}}
}
\eeq

\subsection{Representatives of tangent complexes}
As a preparation of the proof of Theorem \ref{main4}, we find good representatives of tangent complexes starting from the given $\T$-equivariant normal form $\{T\to E\to T^*\}$ \eqref{n2} of $T_{L/U}$. First of all, we find a good representative of $T_L$.

\begin{lem}\label{repTL}
There is an extension $\overline{T}\in \Ext^1(T,p^*T_U)$ on $L$ and a bundle map $\overline{T}\to E$ such that $T_L$ is represented by $\{\overline{T}\to E\to T^*\}$. 
\end{lem}
\begin{proof} By the resolution property on $L$, we can find a quasi-isomorphic complex $\{F_{-1}\to F_0\}$ of $p^*T_U$ with the chain map
$$
\xymatrix{
&&& & T^*   \\
p^*T_U\ar[r] & T_{L/U}[1] &= &F_0\ar[r] &E\ar[u] \\
&&&F_{-1}\ar[r]\ar[u]& T. \ar[u]
}
$$
Taking $\overline{T}:=\mathrm{coker} (F_{-1}\to F_0\oplus T)$, we have all desired properties.
\end{proof}

By \cite{GP}, the decomposition $T_L|_{L^{\T}}=T_L^{\mathrm{fix}}\oplus T_L^{\mathrm{mov}}$ of the restriction of the tangent complex $T_L$ to the fixed locus $L^{\T}$ into the fixed and moving parts gives the induced representatives of $T_{L^{\T}}$ and $T_{L^{\T}/L}[1]$,
$$ 
T_{L\bullet}^{\mathrm{fix}}:=\{\overline{T}^{\mathrm{fix}}\to E^{\mathrm{fix}}\to T^{*\mathrm{fix}}\}\ \ \text{and}\ \ T_{L\bullet}^{\mathrm{mov}}:=\{\overline{T}^{\mathrm{mov}}\to E^{\mathrm{mov}}\to T^{*\mathrm{mov}}\},
$$
respectively. We denote by $T_{L\bullet}^{\mathrm{fix}}$ and $T_{L\bullet}^{\mathrm{mov}}$ the representatives of $T_{L^{\T}}$ and $T_{L^{\T}/L}[1]$, respectively, above. Then 
$$
T_{L\bullet}^{\mathrm{fix}}\ \cong\  \mathrm{Cone}\(T_{L\bullet}^{\mathrm{mov}}\rt{\id}T_{L\bullet}^{\mathrm{mov}}\) \oplus T_{L\bullet}^{\mathrm{fix}}
$$
gives another representative of $T_{L^{\T}}$. This complex has a canonical surjection to $T_{L\bullet}^{\mathrm{fix}}\oplus T_{L\bullet}^{\mathrm{mov}}$ which represents $T_L|_{L^{\T}}$, whose kernel $T_{L\bullet}^{\mathrm{mov}}[-1]$ represents $T_{L^{\T}/L}$. This provides us a short exact sequence of complexes representing the exact triangle $T_{L^{\T}/L}\to T_{L^{\T}} \to T_{L}|_{L^{\T}}$,
\beq{bigrep}
\xymatrix@R=5mm{
0\ar[r]& T^{*\mathrm{mov}}\ar[r] & T^{*\mathrm{mov}} \ar[r] & 0 \ar[r] & 0\\
0\ar[r]& E^{\mathrm{mov}} \ar[u]\ar[r] & E^{\mathrm{mov}}\oplus T^{*\mathrm{mov}}\oplus T^{*\mathrm{fix}} \ar[u] \ar[r] & T^{*\mathrm{mov}}\oplus T^{*\mathrm{fix}} \ar[u]\ar[r] & 0 \\
0\ar[r]& \overline{T}^{\mathrm{mov}} \ar[u]\ar[r] & \overline{T}^{\mathrm{mov}}\oplus E^{\mathrm{mov}}\oplus E^{\mathrm{fix}} \ar[u] \ar[r] & E^{\mathrm{mov}}\oplus E^{\mathrm{fix}} \ar[u]\ar[r] & 0 \\
0\ar[r]& 0 \ar[u] \ar[r] & \overline{T}^{\mathrm{mov}} \oplus \overline{T}^{\mathrm{fix}} \ar[u]\ar[r] & \overline{T}^{\mathrm{mov}} \oplus \overline{T}^{\mathrm{fix}} \ar[u]\ar[r] & 0.
}
\eeq

Using the morphism $\overline{T}|_{L^{\T}}\cong \overline{T}^{\mathrm{mov}} \oplus \overline{T}^{\mathrm{fix}} \to p^*T_U|_{L^{\T}}$ constructed in Lemma \ref{repTL}, we obtain the following Corollary from \eqref{bigrep}.

\begin{cor}\label{grep1} The exact triangle $T_{L^{\T}/L}\to T_{L^{\T}/U} \to T_{L/U}|_{L^{\T}}$ is represented by the following short exact sequence of complexes,
$$ 
\xymatrix@R=5mm{
0\ar[r]& T^{*\mathrm{mov}}\ar[r] & T^{*\mathrm{mov}} \ar[r] & 0 \ar[r] & 0\\
0\ar[r]& E^{\mathrm{mov}} \ar[u]\ar[r] & E^{\mathrm{mov}}\oplus T^{*\mathrm{mov}}\oplus T^{*\mathrm{fix}} \ar[u] \ar[r] & T^{*\mathrm{mov}}\oplus T^{*\mathrm{fix}} \ar[u]\ar[r] & 0 \\
0\ar[r]& \overline{T}^{\mathrm{mov}} \ar[u]\ar[r] & \overline{T}^{\mathrm{mov}}\oplus E^{\mathrm{mov}}\oplus E^{\mathrm{fix}} \ar[u] \ar[r] & E^{\mathrm{mov}}\oplus E^{\mathrm{fix}} \ar[u]\ar[r] & 0 \\
0\ar[r]& 0 \ar[u] \ar[r] & T^{\mathrm{mov}} \oplus T^{\mathrm{fix}} \ar[u]\ar[r] & T^{\mathrm{mov}} \oplus T^{\mathrm{fix}} \ar[u]\ar[r] & 0.
}
$$
In particular, we obtained a representative of $T_{L^{\T}/U}$ in the middle column.
\end{cor}

The morphism $\overline{T}\to p^*T_U$ in Lemma \ref{repTL} also induces $\overline{T}^{\mathrm{mov}}\to p^*T_U^{\mathrm{mov}}|_{L^{\T}}\cong p^*N_{U^{\T}/U}|_{L^{\T}}[-1]$. Letting $N:=p^*N_{U^{\T}/U}|_{L^{\T}}$, we also have the following Corollary.

\begin{cor}\label{grep2} The exact triangle $T_{L^{\T}/U^{\T}}\to T_{L^{\T}/U} \to T_{U^{\T}/U}|_{L^{\T}}$ is represented by the following short exact sequence of complexes,
$$ 
\xymatrix@R=5mm@C=6.5mm{
0\ar[r]& T^{*\mathrm{mov}}\ar[r] & T^{*\mathrm{mov}} \ar[r] & 0 \ar[r] & 0\\
0\ar[r]& E^{\mathrm{mov}}\oplus T^{*\mathrm{mov}}\oplus T^{*\mathrm{fix}} \ar[u]\ar[r] & E^{\mathrm{mov}}\oplus T^{*\mathrm{mov}}\oplus T^{*\mathrm{fix}} \ar[u] \ar[r] & 0 \ar[u]\ar[r] & 0 \\
0\ar[r]& T^{\mathrm{mov}}\oplus E^{\mathrm{mov}}\oplus E^{\mathrm{fix}}  \ar[u]\ar[r] & \overline{T}^{\mathrm{mov}}\oplus E^{\mathrm{mov}}\oplus E^{\mathrm{fix}} \ar[u] \ar[r] & N \ar[u]\ar[r] & 0 \\
0\ar[r]& T^{\mathrm{mov}} \oplus T^{\mathrm{fix}} \ar[u] \ar[r] & T^{\mathrm{mov}} \oplus T^{\mathrm{fix}} \ar[u]\ar[r] & 0 \ar[u]\ar[r] & 0.
}
$$
\end{cor}

\subsection*{One immediate consequence} The localisation theorem \cite[Theorem 3.3(a)]{EG:Lo} tells us that the pushforwards by zero sections  $L^{\T}\into C_{L^{\T}/L}$\vspace{-0.5mm} and $L^{\T}\into \overline{T}^{\mathrm{mov}}$, the moving part of $\overline{T}$ constructed in Lemma \ref{repTL}, induce the isomorphisms in $K$-groups because the zero sections are fixed loci. Corollay \ref{grep1} tells us that the cone $\iota: C_{L^{\T}/L}\into \overline{T}^{\mathrm{mov}}$ lies in $\overline{T}^{\mathrm{mov}}$. Hence the pushforward of $\iota_*$ also induces an isomorphism. So the commutativity of \eqref{localcomm2} is equivalent to that of
\beq{localcomm3}
\xymatrixcolsep{11pc}
\xymatrix{
K\(\MF^{\mathsf{T}}(U, f)\) \ar[r]^-{\iota_*\circ \mathfrak{sp}\circ L^{\T}_{(U,f)}} \ar[d]_-{i_U^*} & K_0^{\mathsf{T}}(\overline{T}^{\mathrm{mov}})_\mathrm{loc} \\
K\(\MF^{\mathsf{T}}(U^{\mathsf{T}}, f^{\mathsf{T}})) \ar[r]^-{\frac{L_{(U^{\T}, f^{\T})}}{\mathfrak{e}_{\T}(N_{U^{\T}/U})\cdot\sqrt{\mathfrak {e}_{\mathsf{T}}}(T_{L/U}^\mathrm{mov})}}  & K_0^{\mathsf{T}}\(L^{\mathsf{T}}\)_\mathrm{loc}. \ar[u]_-{0_{*}}
}
\eeq

\subsection{Deformations of normal cones}
The Lagrangian classes in Theorem \ref{main4} are defined via the normal cones $C_{L^{\T}/U^{\T}}$ and $C_{L/U}$. We explain that the commutativity of the diagram \eqref{localcomm2} ultimately arises from a sequence of relations between these two cones. We denote by a squiggly arrow a deformation over $\AA^1$ interpolating between the generic and the central fibres. Then the relations we consider are
\beq{sqarrows}
\xymatrix{
C_{L/U} \ar@{~>}[r] & C_{L^{\T}/ C_{L/U}} & C_{L^{\T}/U} \ar@{~>}[l] \ar@{~>}[r] & C_{L^{\T}/N_{U^{\T}/U}}&C_{L^{\T}/U^{\T}}. \ar@{_(->}[l]
}
\eeq
Now, we divide the diagram \eqref{localcomm3} into several pieces corresponding to the arrows in \eqref{sqarrows}. The first, second, third and fourth arrows in \eqref{sqarrows} give rise to the diagrams \eqref{1ststep}, \eqref{2ndstep}, \eqref{3rdstep} and \eqref{4thstep} below, respectively, which combine to yield \eqref{localcomm3}. 

\subsection*{First commutativity} Recall the first cone $C_{L/U}$ in \eqref{sqarrows} is used to define the specialisation $\mathrm{sp}$ in Section \ref{Sect:sp}. This cone deforms to the second one $C_{L^{\T}/ C_{L/U}}$ in \eqref{sqarrows} by the deformation $M_1:=M^o_{L^{\T}/C_{L/U}}$\vspace{-0.7mm} corresponding to the first arrow in \eqref{sqarrows}. Along this, we would like to replace $\mathrm{sp}$ with
$$
\mathrm{sp}_1:K\(\MF^{\T} (U, f)\) \To K\(\mf^{\T} (C_{L^{\T}/ C_{L/U}}, f/\lambda^2)\)
$$
using the second cone $C_{L^{\T}/ C_{L/U}}$ in \eqref{sqarrows}. Here, the function $f/\lambda^2$ on $C_{L^{\T}/ C_{L/U}}$ is just the pullback from $C_{L/U}$. Note that the base change of $M_1$ via $L\to C_{L/U}$ is the deformation $M^o_{L^{\T}/L}$\vspace{-0.7mm} defining $\mathfrak{sp}$. So conceptually, we may think of $\mathrm{sp}_1$ as a specialisation of specialisation, $\mathfrak{sp}(\mathrm{sp})$. Precisely, we would like to prove the commutativity of the diagram below
\beq{1ststep}
\xymatrix@C=0.8mm@R=3mm{
& K\(\mf^{\T} (C_{L/U}, f/\lambda^2)\) \ar[rrrrrrrrrr]^-{\cS^{\vir}_E\otimes\sqrt{\det T^*}\otimes \pi^*(-)} &&&&&&&&&& K_0^{\mathsf{T}}(L)_\mathrm{loc} \ar[dd]^-{\mathfrak{sp}}\\
K\(\MF^{\T}(U,f)\) \ar[ru]_-{\mathrm{sp}}\ar[rd]^-{\mathrm{sp}_1} \\
&K\(\mf ^{\T} (C_{L^{\T}/ C_{L/U}}, f/\lambda^2)\) \ar[rrrrrrrrrr]^-{\cS^{\vir}_E\otimes\sqrt{\det T^*}\otimes \pi^*(-)} &&&&&&&&&&K_0^{\mathsf{T}}(C_{L^{\T}/L})_\mathrm{loc}.
}
\eeq
We explain the notations with $\T$-equivariant normal form $\{T\to E\to T^*\}$ of $T_{L/U}$. We denote by $\pi$ both the projection map from the Behrend-Fantechi cone $C\subset E$ to $C_{L/U}\cong [C/T]$\vspace{-0.5mm} and its base change $C_{L^{\T}/C}\to C_{L^{\T}/ C_{L/U}}$ by abusing notation. The twisted Clifford factorisation $\surd{\det T^*}\otimes \cS^{\vir}_E$ on $C_{L^{\T}/C}$ is also defined as the pullback from $C$, hence its support lies in $C_{L^{\T}/L}$.

The upper horizontal in \eqref{1ststep} is $L^{\T}_{(U,f)}$. Hence \eqref{1ststep} suggests a replacement of the upper horizontal of \eqref{localcomm2}. By composing with the pushforward $\iota_*:K_0^{\T}(C_{L^{\T}/L})_{\loc}\to K_0^{\T}(\overline{T}^{\mathrm{mov}})_{\loc}$ in the end, it suggests a replacement of the upper horizontal of \eqref{localcomm3}.

We now explain the construction of $\mathrm{sp}_1$. For $\Fd\in \MF^{\T} (U,f)$, inspired by the construction of $\mathfrak{sp}$, we define $\mathrm{sp}_1(\Fd)$ to be the pullback of $\mathrm{sp}(\Fd)$ if it is locally free. However, since we cannot define a pullback of a coherent factorisation in general, we use the cover $\rho:C_{\Fd}\to C_{L/U}$\vspace{-0.5mm} on which we have the locally free tautological factorisation $\xi_\bullet \in \MF^{\T}(C_{\Fd},f/\lambda^2)$. Note that we have $\rho_*\xi_\bullet= \mathrm{sp}(\Fd)$ \eqref{defsp}. We denote by $\rho^*(-)$ the base change of a space $(-)$ by $\rho$ so that the pullback tautological factorisation $\xi_\bullet$ is defined on $\rho^*(-)$. Inspired by the definition of $\mathrm{sp}(\Fd)$, we define
$$
\mathrm{sp}_1(\Fd):= \rho_*\(\xi_\bullet|_{\rho^*C_{L^{\T}/ C_{L/U}}}\otimes 0^!\cO_{\rho^*M_1}\) \in K\(\mf ^{\T} (C_{L^{\T}/ C_{L/U}}, f/\lambda^2)\).
$$ 
Note that $0^!\cO_{\rho^*M_1}$ equals the restriction to zero of the structure sheaf of the closure of $\rho^*M_1|_{\AA^1-\{0\}}\subset \rho^*M_1$. For the proof of the commutativity of \eqref{1ststep}, we investigate the image of the tautological factorisation $\xi_\bullet$ on $\rho^*M_1$ along the following homomorphism
\begin{align}\label{hahaKKK}
K\(\MF^{\T}(\rho^*M_{1},f/\lambda^2)\) \rt{h\(\rho_*\(\cS^{\vir}_E\otimes\sqrt{\det T^*}\otimes \pi^*(-)\)\)}  K_0^{\T}\(M^o_{L^{\T}/L}\)_{\loc},
\end{align}
where $h(-)$ denotes taking cohomology. Note that $\pi$ denotes the base-change of $\pi:C\to C_{L/U}\cong [C/T]$, hence it is 
$$
\pi:\rho^*M^o_{L^{\T}/C}\to \rho^*M_{1}\cong \rho^*M^o_{L^{\T}/C_{L/U}}.
$$ 
The support of $\cS^{\vir}_E$ is $\rho^*M^o_{L^{\T}/L}$ which is the base-change of $L\into C$. We can check that the homomorphism \eqref{hahaKKK} commutes with the Gysin pullbacks to fibres
\beq{commcommcomm1}
\xymatrix@C=40mm{
K\(\MF^{\T}(\rho^*M_{1},f/\lambda^2)\) \ar[r]^-{h\(\rho_*\(\cS^{\vir}_E\otimes\sqrt{\det T^*}\otimes \pi^*(-)\)\)} \ar[d]_-{\lambda^!} &K_0^{\T}\(M^o_{L^{\T}/L}\)_{\loc} \ar[d]^-{\lambda^!},\\
K\(\MF^{\T}(\rho^*M_{1}|_{\lambda},f/\lambda^2)\) \ar[r]^-{h\(\rho_*\(\cS^{\vir}_E\otimes\sqrt{\det T^*}\otimes \pi^*(-)\)\)} & K_0^{\T}\(M^o_{L^{\T}/L}|_{\lambda}\)_{\loc},
}
\eeq
by the commutativity of $h$ and $\rho_*$ proved in Lemma \ref{OS220} below, that of $\lambda^!$ and $\rho_*$ \cite[Theorem 2.6]{Qu}, that of $\lambda^!$ and $h$ \cite[Lemma 2.15]{OS}, Footnote \ref{OS215}, the property
$$
\lambda^!(E_\bullet\otimes G)=\lambda^*E_\bullet \otimes \lambda^!G
$$
for a locally-free $2$-periodic complex $E_\bullet$ and a coherent sheaf $G$ \cite[Lemma 2.18]{OS} and the commutativity of $\lambda^!$ and $\pi^*$ \cite[Theorem 2.6]{Qu}. Using these, we would like to check that the image of $\xi_\bullet$ is $L^{\T}_{(U,f)}(\Fd)$ at a generic fibre and $\cS^{\vir}_E\otimes\sqrt{\det T^*}\otimes \pi^*(\mathrm{sp}_1(\Fd))$ at the central fibre. Then by definition of $\mathfrak{sp}$, we obtain
$$
\mathfrak{sp}\(L^{\T}_{(U,f)}(\Fd)\)=\cS^{\vir}_E\otimes\sqrt{\det T^*}\otimes \pi^*(\mathrm{sp}_1(\Fd)),
$$
proving the commutativity of \eqref{1ststep}. Indeed, the image at each fibre is
\beq{AAAAA}
h\(\rho_*\(\rho^*(\cS^{\vir}_E\otimes \sqrt{\det T^*})\otimes \pi^*(\xi_\bullet\otimes \lambda^!\cO_{\rho^*M_1})\)\).
\eeq
By the projection formula \cite[Proposition 2.2.10]{CFGKS}, it becomes
$$
h\(\cS^{\vir}_E\otimes\sqrt{\det T^*}\otimes \rho_*\(\pi^*(\xi_\bullet\otimes \lambda^!\cO_{\rho^*M_1})\)\).
$$
Since we have $\rho_*\pi^*=\pi^*\rho_*$ by the smooth base-change theorem \cite[Proposition 5.9]{BDFIK}, the image \eqref{AAAAA} is $L^{\T}_{(U,f)}(\Fd)$ at a generic fibre and $\cS^{\vir}_E\otimes\sqrt{\det T^*}\otimes \pi^*(\mathrm{sp}_1(\Fd))$ at the central fibre.

\subsection*{Commutativity of $h$ and $\rho_*$} We cannot use \cite[Lemma 2.20]{OS} directly for this commutativity unfortunately, but we prove its analogue.

\begin{lem}\label{OS220}
Let $g:Y\to X$ be a proper morphism between stacks and let $Z\subset X$ be a closed substack. Let $G_\bullet \in \mf(Y,0)$ be a coherent $2$-periodic complex on $Y$, absolutely acyclic off $Z\times_X Y$. Then we have
$$
g_*h(G_\bullet)=h(g_*G_\bullet)\ \in\ K_0(Z).
$$
\end{lem} 
\begin{proof} Following \cite[Construction 2.16]{OS}, we can construct a factorisation $\cG_\bullet \in \mf(Y\times \AA^1,0)$ which is flat over $\AA^1$, and the restrictions to the central and generic fibres are isomorphic to $G_\bullet$ and $2$-periodisation of $h(G_\bullet)$, respectively. It is actually the cocone of the morphism, denoted by $f_\bullet$, introduced in between \cite[Eq.(2.13)]{OS} and \cite[Eq.(2.14)]{OS}. From this, we observe that the underlying modules are constant over $\AA^1$. Hence the pushforward $g_*\cG_\bullet \in \mf (X\times \AA^1,0)$ is flat over $\AA^1$, and its restriction to a fibre is the pushforward of the restriction. By applying \cite[Lemma 2.15]{OS}, we know its restriction of the cohomology is the same as the cohomology of the restriction as $K$-group elements, $h(g_*\cG_\bullet)|_t=h((g_*\cG_\bullet)|_t)$, and as discussed above, it is $h(g_*(\cG_\bullet|_t))$. By the $\AA^1$-homotopy invariance of $K$-groups of coherent sheaves, we know $h(g_*\cG_\bullet)|_t$ is independent of $t$. Hence we have
\begin{align*}
h(g_*G_\bullet)=h(g_*\cG_\bullet|_1)&=h(g_*\cG_\bullet)|_1\\
&= h(g_*\cG_\bullet)|_0=h(g_*\cG_\bullet|_0)= h(h(g_*G_\bullet)) = h(g_*G_\bullet).  
\end{align*}
\end{proof}

\subsection*{Second commutativity}
While the first arrow in \eqref{sqarrows} is a well-known deformation, the deformation to the normal cone, the second one from $C_{L^{\T}/U}$ to $C_{L^{\T}/C_{L/U}}$ could be rather unfamiliar. In \cite{KKP}, it is proved that there is a space $M_2$ over $\AA^1$ lying inside the cone stack of the mapping cone of 
$$
T_{L^{\T}/L}\rt{\lambda\cdot\id \oplus i} T_{L^{\T}/L}\oplus T_{L^{\T}/U},
$$
where $i$ denotes the natural map $T_{L^{\T}/L}\to T_{L^{\T}/U}$ and $\lambda$ is the coordinate on $\AA^1$, such that
$$
\lambda^!\cO_{M_2}=\left\{
\begin{array}{cl}
\cO_{C_{L^{\T}/U}} & \text{if }\lambda\neq 0, \\
\cO_{C_{L^{\T}/C_{L/U}}} &\text{if }\lambda=0,
\end{array}
\right.
$$
in the $K$-group of the cone stack, where $\lambda^!$ denotes the Gysin pullback to the fibre $\lambda\in \AA^1$. Moreover, $M_2$ can be taken to satisfy $M_2|_{\lambda}\cong C_{L^{\T}/U}$ at $\lambda\neq 0$ as spaces, not only as $K$-group elements. By Corollary \ref{grep1}, this cone stack lies inside the cokernel of the following morphism of bundle stacks
$$
\overline{T}^{\mathrm{mov}}\rt{\lambda\cdot\id \oplus i} \overline{T}^{\mathrm{mov}}\oplus \left[\frac{ \overline{T}^{\mathrm{mov}} \oplus E^{\mathrm{mov}}\oplus E^{\mathrm{fix}} }{T^{\mathrm{mov}}\oplus T^{\mathrm{fix}}}\right],
$$
where $i$ is the morphism induced by the unique horizontal morphism outward from $\overline{T}^{\mathrm{mov}}$ in the diagram in Corollary \ref{grep1}, by abusing notation. Indeed, this $i$ is $\id\oplus 0\oplus 0$. The functoriality of the construction of \cite{KKP} tells us that there is a morphism from $M_2$ to the deformation constructed by the triple $L\rt{=} L\to U$, which is $C_{L/U}\times \AA^1$. Note that our $M_2$ is constructed by the triple $L^{\T}\to L\to U$. This deformation $C_{L/U}\times \AA^1$ lies in
$$
[E/T]\times \AA^1\cong \mathrm{coker}\left(0 \rt{\lambda\cdot\id \oplus i} 0\oplus \left[\frac{ 0 \oplus E}{T}\right] \right).
$$
and this tells us how $M_2$ factors through $C_{L/U}\times \AA^1$ via the morphism between bundle stacks. Hence, the function $f/\lambda^2$ is defined on $M_2$. The bundle stack has a smooth cover, the cokernel of
$$
\overline{T}^{\mathrm{mov}}\rt{\lambda\cdot\id \oplus i} \overline{T}^{\mathrm{mov}}\oplus \overline{T}^{\mathrm{mov}} \oplus E^{\mathrm{mov}}\oplus E^{\mathrm{fix}} ,
$$
which is the trivial bundle $(\overline{T}^{\mathrm{mov}} \oplus E^{\mathrm{mov}}\oplus E^{\mathrm{fix}} )\times \AA^1$ over $\AA^1$. Since this cover is the base-change of $\pi:E\to [E/T]$ at the central fibre $\lambda=0$, we also denote it by $\pi$ by abusing notation. Then on the pullback
$$
\pi^*M_2\subset \(\overline{T}^{\mathrm{mov}} \oplus E^{\mathrm{mov}}\oplus E^{\mathrm{fix}} \)\times \AA^1,
$$
the quadratic function $q$ is defined via $\pi^*M_2\to \overline{T}^{\mathrm{mov}} \oplus E^{\mathrm{mov}}\oplus E^{\mathrm{fix}}\to E$. The function $f/\lambda^2$ and $q$ are the same because they are the pullbacks from $C=\pi^*C_{L/U}\subset E$ on which they are matched. Along the map $\pi^*M_2\to E$, we can define the pullback Clifford factorisation $\cS^{\vir}_E$ on $\pi^*M_2$ with the support inside $\overline{T}^{\mathrm{mov}}$.

The function $f$ also lies in $I_{L^{\T}/U}^2$ since it is in $I_{L/U}^2$\vspace{-0.5mm} by Lemma \ref{Lem:fq}. Hence we obtained the cover $\mu:C'_{\Fd}\to C_{L^{\T}/U}$ \eqref{MFd} with the tautological factorisation $\xi'_\bullet$. We avoid to use the notation $\rho$ here to keep it for the cover $\rho:C_{\Fd}\to C_{L/U}$ with $\xi_\bullet$. Using the cover $\mu$, we define the specialisation \eqref{defsp} as in Section \ref{Sect:sp},
$$
\mathrm{sp}_2(\Fd)= \mu_*\(\xi'_\bullet\) \in K\(\mf ^{\T} (C_{L^{\T}/ U}, f/\lambda^2)\).
$$
Then the morphisms in the diagram below are all well-defined, and we would like to prove its commutativity.
\beq{2ndstep} 
\xymatrix@C=0.7mm@R=3mm{
& K\(\mf ^{\T} (C_{L^{\T}/ C_{L/U}}, f/\lambda^2)\) \ar[drrrrrrrrrr]^-{\ \ \ \ \ \quad\cS^{\vir}_E\otimes\sqrt{\det T^*}\otimes \pi^*(-)}  \\
K\(\MF^{\T}(U,f)\) \ar[ru]_-{\,\, \mathrm{sp}_1}\ar[rd]^-{\mathrm{sp}_2} &&&&&&&&&&&K_0^{\mathsf{T}}\(\overline{T}^{\mathrm{mov}}\)_\mathrm{loc}. \\
&K\(\mf ^{\T} (C_{L^{\T}/ U}, f/\lambda^2)\) \ar[urrrrrrrrrr]_-{\ \ \ \ \ \ \quad\cS^{\vir}_E\otimes\sqrt{\det T^*}\otimes \pi^*(-)} 
}
\eeq
The proof is similar to that of the commutativity of \eqref{1ststep}, but we do not use the canonical cover $\mu^*M_2$ which is the base-change via \mbox{$M_2 \to C_{L^{\T}/U}\times \AA^1$}\vspace{-0.3mm} by viewing the latter as the deformation of \cite{KKP} constructed by the triple $L^{\T}\to U\rt{=} U$. The problem of using it is then $\xi'_\bullet$ need not be a factorisation of $f/\lambda^2$ there. To define a cover $\overline{M}_2\to M_2$ we want to use, we consider a `double' deformation
$$
D\ :=\ M^o_{L^{\T}\times \AA^1/M^o_{L/U}}\To \AA^1_{\nu}\ \ \text{from}\ \ M^o_{L/U}\To \AA^1_{\lambda}
$$
to the normal cone. Note that our deformation $M_2$ lies in the restriction $D|_{\nu=0}$; and they are the same at $\lambda\neq 0$. Also we have a map
$$
D \To D':=M^o_{L\times \AA^1/M^o_{L/U}}
$$
inducing the map $M_2\to D|_{\nu=0}\to D'|_{\nu=0}\cong C_{L/U}\times \AA^1$. Then except at $\lambda=0$, the cover $\rho^*(D|_{\nu=0})$ contains $\mu^*M_2$ over $\lambda\neq 0$ by its construction as the closure; and $\xi_\bullet$ restricts to $\xi'_\bullet$ along the embedding. So we take 
$$
\overline{M}_2:=\text{closure of }\mu^*M_2\text{ inside }\rho^*(D|_{\nu=0})\ \ \text{and}\ \ \xi'_\bullet:=\text{rest. of }\xi_\bullet.
$$
Then we compute the Gysin pullbacks to fibres of
\beq{BBBBB}
h\(\mu_*\(\cS^{\vir}_E\otimes \sqrt{\det T^*}\otimes \pi^*(\xi'_\bullet\otimes \cO_{\overline{M}_2})\)\).
\eeq
The parallel argument with the proof of the commutativity \eqref{commcommcomm1} between \eqref{hahaKKK} and the Gysin pullbacks to fibres shows the following equality,
$$
\lambda^!\eqref{BBBBB}=h\(\mu_*\(\cS^{\vir}_E\otimes \sqrt{\det T^*}\otimes \pi^*(\xi'_\bullet\otimes \lambda^! \cO_{\overline{M}_2})\)\).
$$
Then by the projection formula \cite[Proposition 2.2.10]{CFGKS} and the commutativity $\mu_*\pi^*=\pi^*\mu_*$ \cite[Proposition 5.9]{BDFIK}, the equation becomes
$$
h\(\cS^{\vir}_E\otimes \sqrt{\det T^*}\otimes \pi^*\mu_*(\xi'_\bullet\otimes \lambda^!\cO_{\overline{M}_2})\).
$$
At a generic fibre, it is
$$
\cS^{\vir}_E\otimes \sqrt{\det T^*}\otimes \pi^*\mathrm{sp}_2(\Fd),
$$
which is the same as the Gysin pullback to the central fibre by the $\AA^1$-homotopy invariance of the $K$-groups of coherent sheaves. So it remains to show that the pullback to the central fibre is
\beq{BBBBB1}
\cS^{\vir}_E\otimes \sqrt{\det T^*}\otimes \pi^*\mathrm{sp}_1(\Fd),
\eeq
which proves the commutativity of \eqref{2ndstep}. We claim that
\beq{Lem620}
\xi_\bullet\otimes 0^!\cO_{\rho^*M_1}\=\xi'_\bullet\otimes 0^!\cO_{\overline{M}_2}.
\eeq
We prove this claim in an appropriate $K$-group. Once this claim holds true, their pushdowns to $D|_{(0,0)}$ pull back by $\pi^*$ to become the equality $\lambda^!\eqref{BBBBB}=\eqref{BBBBB1}$ in $K_0^{\T}(\overline{T}^{\mathrm{mov}})_{\loc}$. So we also need to explain that there exists a proper pushdown.

The key idea is to find a space $\overline{D}$ over $\AA^1_\lambda\times \AA^1_\nu$ mapping to $D$ containing both $\rho^*M_1$ and $\overline{M}_2$. To define it we first consider the map $D\to M^o_{L/U}\times \AA^1_\nu$ obtained by the fact that $D$ is the deformation to the normal cone from $M^o_{L/U}$. Using $\rho: M^o_{\Fd}\to M^o_{L/U}$, we take $\rho^*D$ which coincides with 
$$
\rho^*D\cong
\left\{
\begin{array}{cl}
D & \text{at } \lambda\neq 0, \\
M^o_{\Fd}\times \AA^1_{\nu} & \text{at }\nu \neq 0.
\end{array}
\right.
$$ 
It is because $\rho$ is $\id$ at $\lambda\neq 0$ and $D\to M^o_{L/U}\times \AA^1_\nu$\vspace{-0.7mm} is $\id$ at $\nu\neq 0$. We also consider a different map $D\to M^o_{L\times \AA_\lambda^1/U\times \AA^1_\lambda}\cong M^o_{L/U}\times \AA^1_\lambda$\vspace{-0.5mm} constructed by using the functorial property of deformations. Composing it with $\rho^*D\to D$ gives $\rho^*D\to M^o_{L/U}\times \AA^1_{\lambda}$. Pulling $\rho^*D$ back via $\rho$ again, we have
$$
\rho^*(\rho^*D)\ \cong\
\left\{
\begin{array}{cl}
\rho^*D & \text{at } \lambda\neq 0, \\
M^o_{\Fd}\times \AA^1_{\nu} & \text{at }\nu \neq 0,
\end{array}
\right.
$$
because $D\to M^o_{L/U}\times \AA^1_{\lambda}$\vspace{-0.5mm} is $\id$ at $\lambda\neq 0$ and $\rho$ is $\id$ at $\nu\neq 0$ this time. Then take $\overline{D}$ to be the closure of $\rho^*(\rho^*D)|_{\lambda\neq 0, \nu\neq 0}$ inside $\rho^*(\rho^*D)$.

Then we prove \eqref{Lem620} in the $K$-group $K\(\mf^{\T}(\overline{D}|_{(0,0)}, f/\lambda^2)\)$. The fibre $\overline{D}|_{(0,0)}$ projects to $D|_{(0,0)}$. Indeed, this $\overline{D}$ contains both $\rho^*M_1$ and $\overline{M}_2$ at $\lambda=0$ and $\nu=0$, respectively, and they are matched at $\nu\neq 0$ and $\lambda\neq 0$, respectively. On the fibre $\overline{D}|_{(0,0)}$, we have the tautological factorisation $\xi_\bullet$. Note that the two factorisations at $(0,0)$ defined by using $\rho_\lambda$ and $\rho_\nu$ are isomorphic. It restricts to $\xi_\bullet$ on $\rho^*M_1|_{\nu=0}$ and $\xi'_\bullet$ on $\overline{M}_2|_{\lambda=0}$. Hence, it is sufficient to prove that $0^!\cO_{\rho^*M_1}=0^!\cO_{\overline{M}_2} \in K_0^{\T}(\overline{D}|_{(0,0)})$. This holds true, 
$$
0_\nu^!\cO_{\rho^*M_1}=0_\nu^!0_{\lambda}^!\cO_{\overline{D}}=0_{\lambda}^!0_\nu^!\cO_{\overline{D}}=0^!_{\lambda}\cO_{\overline{M}_2}.
$$
The first equality holds because $\rho^*M_1$ coincides with $\overline{D}|_{\lambda=0}$ at $\nu\neq 0$. Similarly, the third equality holds because $\overline{M}_2$ coincides with $\overline{D}|_{\nu=0}$ at $\lambda\neq 0$.

\subsection*{Third commutativity} To simplify the notation, we denote by $N=N_{U^{\T}/U}$ the normal bundle. Since the cone $C_{L^{\T}/N}$ is isomorphic to $C_{L^{\T}/U^{\T}}\oplus N|_{L^{\T}}$, we can define
\beq{sp3}
\mathrm{sp_3}:K\(\MF^{\T}(U,f)\)\To K\(\mf^{\T}(C_{L^{\T}/U},f/\lambda^2)\), \ \mathrm{sp}_3(\Fd):= \mathrm{pr}_N^*\mathrm{sp}_{\T}(i_U^*\Fd)
\eeq
using the smooth projection map $\mathrm{pr}_N:C_{L^{\T}/N}\to C_{L^{\T}/U^{\T}}$. In case there is no confusion, we abuse the notation $N$ for the restriction $N|_{L^{\T}}$ so that we write the normal cone of $T^{\mathrm{mov}}$ into $\overline{T}^{\mathrm{mov}}$ as a direct sum $T^{\mathrm{mov}}\oplus N$ with the simple notation. Then we would like to prove the commutativity of
\beq{3rdstep}
\xymatrix@C=0.6mm@R=3mm{
& K\(\mf ^{\T} (C_{L^{\T}/ U}, f/\lambda^2)\) \ar[rrrrrrrrrr]^-{\cS^{\vir}_E\otimes\sqrt{\det T^*}\otimes \pi^*(-)} &&&&&&&&&&K_0^{\mathsf{T}}(\overline{T}^{\mathrm{mov}})_\mathrm{loc} \ar[dd]^{\cong}_-{\mathfrak{sp}}\\
K\(\MF^{\T}(U,f)\) \ar[ru]_-{\mathrm{sp}_2}\ar[rd]^-{\mathrm{sp}_3} \\
&K\(\mf ^{\T} (C_{L^{\T}/ N}, f/\lambda^2)\) \ar[rrrrrrrrrr]^-{\cS^{\vir}_E\otimes\sqrt{\det T^*}\otimes \pi^*(-)} &&&&&&&&&&K_0^{\mathsf{T}}(T^{\mathrm{mov}}\oplus N)_\mathrm{loc} .
}
\eeq
For the proof, we use the third deformation $M_3$ which is again constructed by the method of \cite{KKP}, from $C_{L^{\T}/U}$ to $C_{L^{\T}/N}$ corresponding to the third arrow in \eqref{sqarrows}. It lies in the cone stack of the cocone of
$$
T_{L^{\T}/U}\oplus T_{U^{\T}/U}\rt{j-\lambda\id} T_{U^{\T}/U},
$$
where $j:T_{L^{\T}/U}\to T_{U^{\T}/U}$ is the natural morphism from $L^{\T}\to U^{\T}\to U$. By Corollary \ref{grep2}, this cone stack lies inside the kernel of the following morphism of bundle stacks
$$
\left[\frac{ \overline{T}^{\mathrm{mov}} \oplus E^{\mathrm{mov}}\oplus E^{\mathrm{fix}} }{T^{\mathrm{mov}}\oplus T^{\mathrm{fix}}}\right]\oplus N\rt{j-\lambda\id} N,
$$
where the morphism $j$ is induced by the unique horizontal morphism toward to $N$ in the diagram in Corollary \ref{grep2} by abusing notation. This $j$ is induced by $\overline{T}^{\mathrm{mov}}\to N$ defined in Lemma \ref{repTL}. Again, the functoriality of the construction of \cite{KKP} tells us that there is a morphism from $M_3$ to the deformation constructed by the triple $L\to U\to U$, which is $C_{L/U}\times \AA^1$. It lies in
$$
[E/T]\times \AA^1\ \cong\ \ker\left(\left[\frac{0\oplus E}{T}\right]\oplus 0\rt{j-\lambda\id} 0 \right).
$$
It tells us how $M_3$ factors through $C_{L/U}\times \AA^1$ via the morphism between bundle stacks. Hence, the function $f/\lambda^2$ is defined on $M_3$ and this function is matched with the one on $M_2$ at generic fibres. The bundle stack has a smooth cover, the kernel of
$$
\overline{T}^{\mathrm{mov}} \oplus E^{\mathrm{mov}}\oplus E^{\mathrm{fix}}\oplus N\rt{j-\lambda\id} N,
$$
which is isomorphic to
$$
\left\{
\begin{array}{cl}
\overline{T}^{\mathrm{mov}} \oplus E^{\mathrm{mov}}\oplus E^{\mathrm{fix}} & \text{if } \lambda\neq 0, \\
T^{\mathrm{mov}} \oplus E^{\mathrm{mov}}\oplus E^{\mathrm{fix}} \oplus N & \text{if } \lambda=0.
\end{array}
\right.
$$
Again we denote this cover by $\pi$ by abusing notation since it is the base-change of $\pi:E\to [E/T]$ at the central fibre. Then we have a morphism $\pi^*M_3\subset\text{the bundle stack}\to E$ under which we define the quadratic function $q$ and the Clifford factorisation $\cS^{\vir}_E$ as pullbacks. The functions $f/\lambda^2$ and $q$ on $\pi^*M_3$ are the same because they are pullbacks from $C\subset E$. The support of $\cS^{\vir}_E$ lies in $\ker(\overline{T}^{\mathrm{mov}}\oplus N\to N)$ which is the deformation to the normal cone of $T^{\mathrm{mov}}$ into $\overline{T}^{\mathrm{mov}}$. 

A crucial difference of $M_3$ from $M_2$ is that the map $M_3\to C_{L/U}\times \AA^1$ factors through $C_{L^{\T}/U}\times \AA^1$, so that we can use the base-change $\mu^*M_3$ by $\mu:C'_{\Fd}\to C_{L^{\T}/U}$. To prove the commutativity of \eqref{3rdstep}, we compute the Gysin pullbacks to fibres of
\beq{CCCCC}
h\(\mu_*\(\cS^{\vir}_E\otimes \sqrt{\det T^*}\otimes \pi^*(\xi'_\bullet\otimes \cO_{\mu^*M_3})\)\),
\eeq
which is a $K$-group element inside the deformation from $\overline{T}^{\mathrm{mov}}$ to $T^{\mathrm{mov}}\oplus N$. The parallel argument with the proof of the commutativity \eqref{commcommcomm1} shows that its Gysin pullback to the generic fibre is $h(\cS^{\vir}_E\otimes \sqrt{\det T^*}\otimes \pi^*(\mathrm{sp}_2(\Fd)))$. It remains to show that its pullback to the central fibre is 
\beq{CCCCC1}
h(\cS^{\vir}_E\otimes \sqrt{\det T^*}\otimes \pi^*(\mathrm{sp}_3(\Fd))).
\eeq 
That requires to show that $\mu_*(\xi'_\bullet \otimes 0^!\cO_{\mu^*M_3})=\mathrm{sp}_3(\Fd)$ which is induced by
\begin{equation*}
\label{CCCCC2}
\xi'_\bullet \otimes 0^!\cO_{\mu^*M_3}=\mathrm{pr}_N^*\(\xi^{\T}_\bullet \otimes \cO_{C_{i_U^*\Fd}}\),
\end{equation*}
providing the equality $0^!\eqref{CCCCC}=\eqref{CCCCC1}$ in $K_0^{\T}(T^{\mathrm{mov}}\oplus N)_{\loc}$. The proof is similar to that of \eqref{Lem620}.

We again consider the double deformation
$$
D\ :=\ M^o_{L^{\T}\times \AA^1/M^o_{U^{\T}/U}}\To \AA^1_{\nu}\ \ \text{from}\ \ M^o_{U^{\T}/U}\To \AA^1_{\lambda}
$$
to the normal cone, so that $M_3$ lies in the restriction $D|_{\nu=0}$ and equals that on $\lambda\neq 0$. The deformation $M^o_{L^{\T}/N}$ lies in the restriction $D|_{\lambda=0}$ and equals that on $\nu\neq 0$. Note that $f$ lies in $I_{L^{\T}\times \AA^1/M^o_{U^{\T}/U}}^2$ so that we can define
$$
\mu: \overline{D}:=M^o_{\Fd}\To D\ \ \text{defining}\ \ \mathrm{sp}\ \ \text{in}\ \ \eqref{defsp},
$$
by abusing notation $\mu$. Then $\mu^*M_3$ lies in $\overline{D}|_{\nu=0}$ and equals that on $\lambda\neq 0$. The pullback $\mathrm{pr}_N^*M^o_{i_U^*\Fd}$ of the cover $\mu_{\T}:M^o_{i_U^*\Fd}\to M^o_{L^{\T}/U^{\T}}$\vspace{-1.0mm} lies in $\overline{D}|_{\lambda=0}$ and equals that on $\nu\neq 0$.  The tautological factorisation $\overline{\xi}_\bullet$ restricts to $\xi'_\bullet$ on $\mu^*M_3$ and to $\mathrm{pr}_N^*\xi^{\T}_\bullet$ on $M^o_{\Fd}$ by their construction as the closures. Hence, it is sufficient to prove that $0^!\cO_{\mu^*M_3}=\mathrm{pr}_N^*\cO_{C_{i_U^*\Fd}} \in K_0^{\T}(\overline{D}|_{(0,0)})$. This holds true, 
$$
0_\lambda^!\cO_{\mu^*M_3}=0_\lambda^!0_{\nu}^!\cO_{\overline{D}}=0_{\nu}^!0_\lambda^!\cO_{\overline{D}}=0^!_{\nu}\mathrm{pr}_N^*\cO_{M^o_{i_U^*\Fd}}=\mathrm{pr}_N^*\cO_{C_{i_U^*\Fd}}.
$$

\subsection*{Fourth commutativity}
Now using the fourth arrow $i_4:C_{L^{\T}/U^{\T}}\into C_{L^{\T}/N}$ in \eqref{sqarrows}, we prove the commutativity of the diagram,
\beq{4thstep}
\xymatrix@C=3.0mm@R=10mm{
K\(\MF^{\T}(U,f)\) \ar[r]^-{\mathrm{sp}_3} \ar[d]_-{i_U^*} & K\(\mf ^{\T} (C_{L^{\T}/ N}, f/\lambda^2)\) \ar[rrrr]^-{\cS^{\vir}_E\otimes\sqrt{\det T^*}\otimes \pi^*}&&&&K_0^{\mathsf{T}}(T^{\mathrm{mov}}\oplus N)_\mathrm{loc} \\
K\(\MF^{\T}(U^{\T},f^{\T})\) \ar[r]^-{\mathrm{sp}_{\T}} &K\(\mf ^{\T} (C_{L^{\T}/ U^{\T}}, f/\lambda^2)\) \ar[rrrr]^-{\frac{\cS^{\vir}_{E^{\mathrm{fix}}}\otimes\sqrt{\det T^{\mathrm{fix}*}}\otimes \pi^*_{\T}}{\mathfrak{e}_{\T}(N_{U^{\T}/U})\cdot\sqrt{\mathfrak {e}_{\mathsf{T}}}(T_{L/U}^\mathrm{mov})}}\ar[u]^-{\mathrm{pr}_N^*} &&&&K_0^{\mathsf{T}}(L^{\T})_\mathrm{loc}. \ar[u]_-{i_{*}}
}
\eeq
Here, $i:L^{\T}\into T^{\mathrm{mov}}\oplus N$ denotes the embedding of the zero section. The first square is commutative by definition of $\mathrm{sp}_3$ \eqref{sp3}. For the second square, we chase an arbitrary coherent factorisation $G_\bullet\in \mf^{\T}(C_{L^{\T}/U^{\T}},f/\lambda^2)$. Since $E|_{L^{\T}}$ and $T|_{L^{\T}}$ split into $E^{\mathrm{mov}}\oplus E^{\mathrm{fix}}$ and $T^{\mathrm{mov}}\oplus T^{\mathrm{fix}}$, respectively, we get
\begin{align}\label{reallylast}
& \cS^{\vir}_E\otimes\sqrt{\det T^*}\otimes \pi^*\(\mathrm{pr}_N^*(G_\bullet)\) \\
&= \cS^{\vir}_{E^{\mathrm{mov}}} \otimes \sqrt{\det T^{\mathrm{mov}*}}\otimes \cS^{\vir}_{E^{\mathrm{fix}}} \otimes\sqrt{\det T^{\mathrm{fix}*}}\otimes \mathrm{pr}_N^*\pi_{T^{\mathrm{mov}}}^*\(\pi_{\T}^*(G_\bullet)\), \nonumber
\end{align}
where $\pi_{T^{\mathrm{mov}}}$ denotes the base-change of $E^{\mathrm{mov}}\to [E^{\mathrm{mov}}/T^{\mathrm{mov}}]$ whose composition with that of $\pi_{\T}:E^{\mathrm{fix}}\to [E^{\mathrm{fix}}/T^{\mathrm{fix}}]$ becomes that of $\pi:E\to$~$[E/T]$. Then the morphism $\mathrm{pr}_N^*\pi_{T^{\mathrm{mov}}}^*$ is the inverse of $i^*$, which is again the inverse of $i_*/\mathfrak{e}_{\T}(T^{\mathrm{mov}}\oplus N)$ by the self-intersection formula. Since the pullback space
\begin{align*}
\mathrm{pr}_N^*\pi_{T^{\mathrm{mov}}}^*\(\pi_{\T}^*(C_{L^{\T}/U^{\T}})\)&\cong T^{\mathrm{mov}}\oplus 0 \oplus C_{\T}\oplus N\\
& \subset T^{\mathrm{mov}}\oplus E^{\mathrm{mov}}\oplus E^{\mathrm{fix}}\oplus N
\end{align*}
does not meet $E^{\mathrm{mov}}$, the tautological section on $E^{\mathrm{mov}}$ becomes zero on the space. Hence $\cS^{\vir}_{E^{\mathrm{mov}}} $ in \eqref{reallylast} equals $\surd{\mathfrak{e}}_{\T}(E^{\mathrm{mov}})$ \cite[Page 56]{OT1}. Therefore, using $\surd{\mathfrak{e}}_{\T}(T^{\mathrm{mov}}_{L/U})=\mathfrak{e}_{\T}(T^{\mathrm{mov}})\cdot \surd{\det T^{\mathrm{mov}}}/\surd{\mathfrak{e}}_{\T}(E^{\mathrm{mov}})$, \eqref{reallylast} becomes 
$$
i_*\left(\frac{\cS^{\vir}_{E^{\mathrm{fix}}}\otimes\sqrt{\det T^{\mathrm{fix}*}}\otimes \pi^*_{\T}(G_\bullet)}{\mathfrak{e}_{\T}(N_{U^{\T}/U})\cdot\sqrt{\mathfrak {e}_{\mathsf{T}}}(T_{L/U}^\mathrm{mov})}\right),
$$
proving the commutativity of \eqref{4thstep}.

\appendix
\section{Spin structure and spinor bundle} \label{AppA}
\subsection*{Spin structure} 
Let $\rho:\mathrm{Spin}(2n) \to SO(2n)$ denotes the universal $(\Z/2)$-covering. A {\em spin structure} on a $SO(2n)$-bundle $E$ on $Y$ is a choice of $(\Z/2)$-equivariant lift of the principal frame bundle $\P_{SO}(E)$. i.e., a commutative diagram
$$
\xymatrix{
\P_{\mathrm{Spin}}(E) \ar[r]^\xi \ar[dr]^\pi & \P_{SO}(E) \ar[d] \\
& Y
}
$$
such that $\pi$ is a principal $\mathrm{Spin}(2n)$-bundle and $\xi$ satisfies $\xi(e\cdot g) = \xi(e)\cdot\rho(g)$. When we say $(E, q)$ is spin, it is understood as a $SO$-bundle with a preferred choice of spin structure.

\subsection*{Spinor bundle} 
A spin bundle $(E, q)$ comes with the \emph{spinor bundle} denoted by $\cS_E$. It is well-known that the spin group $\mathrm{Spin}(2n)$ is a subgroup of the standard Clifford algebra $Cl(\C^{2n}, q)$ (with multiplication). Hence the unique nontrivial irreducible Clifford module $S$ of $Cl(\C^{2n},q)$ is acted on by $\mathrm{Spin}(2n)$. The spinor bundle $\cS_E$ is then defined to be
$$
\cS_E : = \P_{\mathrm{Spin}}(E)\times _{\mathrm{Spin}(2n)} S.
$$
It is an irreducible Clifford module of $Cl(E,q)$. A $(\Z/2)$-grading on $\cS_E$ is determined by the orientation on $E$. Let $\{e_1, \ldots, e_{2n}\}$ be a {\em positive} orthonormal local frames of $E$ giving the orientation $\omega= i^n e_1\cdots e_{2n}$. Since its square is $\id$, we obtain an eigen-decomposition
$$
\cS_E \= \cS_{E}^+\;\oplus\; \cS_{E}^-.
$$

\subsection*{Pfaffian lines and square root determinant line bundle}
When $(E,q)$ is spin, we can assign the square root of determinant line bundle $\surd{\det \Lambda}$
to each maximal isotropic subbundle $\Lambda \subset E$. For each pair of a point $y\in Y$ with a maximal isotropic subspace $\Lambda_y \subset E_y$, we define its \emph{Pfaffian line} 
$$
\mathrm{Pf}(\Lambda_y) := \left\{v\in (\cS_E)_y : \Lambda_y \cdot v \equiv 0\right\}.
$$
To check that $\mathrm{Pf}(\Lambda_y)$ is a line, choose a basis $\{\lambda_1, \ldots, \lambda_n, \lambda^*_1, \ldots, \lambda^*_n\}$ of $E_y$ so that $\Lambda _y = \langle \lambda_1, \ldots, \lambda_n\rangle$ and $q(\lambda_i, \lambda_j^*)=\delta_{ij}$. The Clifford relation $\lambda_i\lambda_i^*+\lambda_i^*\lambda_i  =1$ implies that $(\lambda_i\lambda_i^*, \lambda_i^*\lambda_i )$ is an idempotent pair. Therefore we obtain a decomposition of $(\cS_E)_y$ by two components:
$$
(\cS_E)_y = (\lambda_i\lambda_i^*) \cdot (\cS_E)_y \oplus (\lambda_i^*\lambda_i) \cdot (\cS_E)_y.
$$
The Clifford relation tells us that $(\lambda_i \cdot)$ annihilates the first component, and sends the second component isomorphically to the first one. Doing this process inductively, we conclude that 
$$
\mathrm{Pf}(\Lambda_y) = \prod_{i=1}^n(\lambda_i\lambda_i^*) \cdot (\cS_E)_y. 
$$
Each $(\lambda_i\lambda_i^*)$ halves the dimension of $(\cS_E)_y$, so $\mathrm{Pf}(\Lambda_y)$ must be a line. 
This assignment defines a morphism
$$
\mathrm{Pf} : \mathrm{IGr}_Y(n, E) \to \mathbb P_Y(\cS_E),\hspace{10pt}(y, \Lambda_y) \mapsto \(y, \mathrm{Pf}(\Lambda_y)\),
$$
which fits into a commutative diagram 
\beq{IGrdia}
\xymatrix{
\mathrm{IGr}_Y(n, E) \ar[r]^{\mathrm{Pf}} \ar[dr]& \mathbb P_Y(\cS_E)\ar[d] \\
& Y \ar@/^1pc/[ul]^\Lambda
}
\eeq
satisfying $\mathrm{Pf}^*\cO_{\mathbb P_Y(\cS_E)}(-2) \simeq \det {\varLambda}$ \cite[Proposition 12.3.1]{PS}, where $\varLambda$ is the universal isotropic bundle on $\mathrm{IGr}_Y(n, E)$. Hence $\mathrm{Pf}^*\cO_{\mathbb P_Y(\cS_E)}(-1)$ can be taken to be the square root determinant line bundle $\surd{\det \varLambda}$.

Since an isotropic subbundle $\Lambda \subset E$ gives a section of the isotropic Grassmannian bundle as in \eqref{IGrdia}, its square root determinant line bundle can be defined as 
$$
\sqrt{\det\Lambda}\ :=\ (\Lambda\circ \mathrm{Pf})^*\cO_{\mathbb P_Y(\cS_E)}(-1).
$$

\subsection*{Proof of Proposition \ref{prop: sqcorrection}}
Now we are ready to prove Proposition \ref{prop: sqcorrection}.
\begin{proof}[Proof of Proposition \ref{prop: sqcorrection}]
Since $\surd{\det \Lambda}$ is a line in $\cS_E$, there is a canonical morphism
$$
Cl(E, q) \To \sqrt{\det \Lambda}^{-1}\otimes \cS_E.
$$ 
Its kernel contains the left ideal generated by $\Lambda$ by definition of the Pfaffian line. We can check it is exactly the kernel since the kernel must be isotropic. Hence the morphism is surjective by counting dimensions. This induces an isomorphism (of ungraded modules) from the quotient
\beq{eq: spiniso}
\cS_\Lambda  \rt{\cong} \sqrt{\det \Lambda}^{-1}\otimes \cS_E.
\eeq
Hence $\cS_\Lambda \otimes \sqrt{\det \Lambda}$ is independent of $\Lambda$. Using a local basis of $\Lambda$, one can check that $\omega: \cS_{\Lambda}\to \cS_{\Lambda}$ takes $\Lambda^{\mathrm{even}}\Lambda^*$ to $(-1)^{|\Lambda |}\Lambda^{\mathrm{even}}\Lambda^*$ (and takes $\Lambda^{\mathrm{odd}}\Lambda^*$ to $(-1)^{|\Lambda |}\Lambda^{\mathrm{odd}}\Lambda^*$). So \eqref{eq: spiniso} is an isomorphism after shifting by $(-1)^{|\Lambda|}$.
\end{proof}

\subsection*{Without spin}
Even if $E$ is not spin, fixing one maximal isotropic subbundle $\Lambda_0\subset E$, we have a similar diagram of \eqref{IGrdia} replacing $\cS_E$ by $\cS_{\Lambda_0}$,
$$
\xymatrix{
\mathrm{IGr}_Y(n, E) \ar[r]^{\mathrm{Pf}_0} \ar[dr]& \mathbb P_Y(\cS_{\Lambda_0})\ar[d] \\
& Y 
}
$$
satisfying $\mathrm{Pf}_0^*\cO_{\mathbb P_Y(\cS_{\Lambda_0})}(-2) \simeq \det {\varLambda}\otimes \det \Lambda^*_0$. Here for another maximal isotropic subbundle $\Lambda\subset E$, there is a line bundle on $Y$ whose square is $\det {\Lambda}\otimes \det \Lambda^*_0$.

\bibliographystyle{halphanum}

\bigskip \noindent {\tt{dwchoa@chungbuk.ac.kr}} \medskip

\noindent Department of Mathematics, Chungbuk National University, Cheongju 28644, Korea

\bigskip \noindent {\tt{jeongseok@snu.ac.kr}} \medskip

\noindent Department of Mathematical Sciences and Research Institute of Mathematics, Seoul National University, Seoul 08826, Korea



\end{document}